\documentclass[a4paper,12pt]{article}
\usepackage{amssymb,amsmath}
\usepackage[dvips]{color}
\usepackage[dvips]{graphics}
\usepackage[dvips]{color}
\usepackage[dvips]{graphics}

\newtheorem{theorem}{Theorem}[section]
\newtheorem{lemma}{Lemma}[section]
\newtheorem{proposition}{Proposition}[section]

\newtheorem{corollary}{Corollary}[section]

\begin{document}
\title{Infinitely dimensional, affine nil algebras $A\otimes A^{op}$ and $A\otimes A$ exist}
\author{Agata Smoktunowicz\thanks{ School of Mathematics, University of Edinburgh,  A.Smoktunowicz@ed.ac.uk} 
}

\pagestyle{myheadings}

\maketitle

\begin{abstract} In this paper we answer two questions from \cite{p}, by showing that, over any algebraically closed field, $K$, there is a finitely generated, infinitely dimensional algebra $A$ such that algebras  $A\otimes _{K}A$ and $A\otimes _{K} A^{op}$  are nil.
\end{abstract}

 
\section{Introduction}

 Several authors have studied the properties of tensor products of algebras over a field; for example,  in \cite{bokut}, \cite{bergman}, \cite{ferrero}, \cite{jespers}, \cite{krempa}, \cite{lawrenc},\cite{niholson} and \cite{rescosmall}. 
  In 1993, the question was posed  in \cite{p}  as to whether every affine algebra  $A$, such that $A\otimes A$  is nil,  is locally nilpotent. In \cite{p2}  it was shown  that this is indeed the case for algebras over ordered fields; however, the general case remained open.  In \cite{clp}, a related question  was asked, namely whether a Jacobson radical of algebras $A\otimes A$ and $A\otimes A^{op}$ is always locally nilpotent (recall that a nil ring is always Jacobson radical).  The aim of this paper is to show that in general these two statements need not hold; in other words, that the following holds.
\begin{theorem}\label{1}
 Over every algebraically closed field there is a finitely generated graded $K$-algebra $R$ such that $R\otimes _{K} R$ and $R\otimes _{K} R^{op}$ are nil algebras and  $R$ is an infinitely dimensional algebra.
\end{theorem} 

 The properties of algebras $A\otimes A^{op}$ may be pertinent to the study of $A$-bimodules. Recall that $A$-bimodules correspond to $A^{e}$-modules, where $A^{e}=A\otimes A^{op}$.  Determining the Jacobson radical or nil radical of algebra $A\otimes A^{op}$ may be useful in the context of the Nakayama lemma \cite{lam}.

 Quite often, the structure of the tensor product of an algebra with itself or with its opposite provides information about the structure of the algebra. For example, if $A$ is a finite dimensional $F$-algebra such that $A \otimes A^{op}$ is a matrix $F$-algebra, then $A$ is central simple. In this case, if $A \otimes A$ is a matrix $F$-algebra then $A$ is a central simple of exponent $2.$ 
Several new directions have recently developed within  the area of infinitely dimensional Brauer Groups. The methods developed in this paper might be related to these results, as in many cases they permit the construction of infinite dimensional 
graded algebras $A$ such that $A\otimes A^{op}$  
 satisfy a restricted number of  prescribed relations (however, in general these constructed  algebras are not simple).  
 Another application could be to construct various examples of nilpotent algebras; for some open questions in this area and connections with group theory, see \cite{kazarin}.  Another application might  be to study the properties of polynomial rings which are clean, using the results of Kanwar, Leroy and Matczuk \cite{lm}, or to obtain examples of rings with clean tensor product other than those in \cite{glasgow}.
  
The construction from this paper can also be applied to some other situations where we need to construct algebras $A$ such that $A\otimes A$ and $A\otimes A^{op}$ satisfy prescribed relations; for example, it could be applied to study local rings, or rings with  local tensor products $A\otimes A$ and $A\otimes A^{op}$ (for some related results, see \cite{lawrence}). It may also be useful for constructing examples of bialgebras or Hopf algebras satisfying some properties, for example  of infinite dimensional not pointed, nil Hopf algebras; however, this seems complex. 

\section{ Notation and plan }
 Let $K$ be a field, and let $A$ be a free algebra over $K$ with two generators $x,y$ of degree one.  Obviously, the monomials of the form $x_{i_1}x_{i_2}\dots x_{i_n}$ 
where $x_{i_{1}}, \ldots , x_{i_{n}}\in \{x,y\}$  form a $K$-basis of $A$.
By $A^1$ we denote the algebra obtained from $A$ by the adjunction of a unity. If an element $a\in A$ is a sum  of monomials of the same degree multiplied by coefficients, then by $deg(a)$ we mean the number expressing the common degree of the monomials.
 For a positive integer $n$ and a subset $S$ of $A$ by $S(n)$ we denote the set of all elements of $S$
  which are sums of monomials of length equal to $n$ multiplied by elements from $K$. So $A=A(1)+A(2)+\ldots .$ By a slight abuse of  notation,  $A(0)$ will denote $K$ (in fact it means $K\cdot 1_{A^{1}}$). Finally, by $\mathcal{M}$ we denote the set of all monomials of $A$.
 We put $A(\alpha )=0$ if $\alpha <0$. An ideal in $A$ will mean a two-sided ideal in $A$.
   Let $S$ be a linear space over a field $K$, then $\dim _{K} S$ will denote the dimension of $S$. 
 We say that element $a\in A\otimes A$ is homogeneous if $a\in A(m)\otimes A(n)$ for some $m, n$.
  We will then  say that the left degree of $a$ is $m$, denoted $l(a)=m$, and the right degree of $a$ is $n$, denoted as  $r(a)=n$. 
 A linear subspace of $A\otimes A$ is homogeneous if it is generated by homogeneous elements. 

 We will use the following notation throughout this paper.

Let $1\leq \alpha $, $0< n_{1}<n_{2}$ be natural numbers and let $F\subseteq A(n_{1})\otimes A(\alpha ),$ $G\subseteq A(\alpha ) \otimes A(n_{1})$. We define   
 \[Q(F, n_{1}, n_{2})=\sum_{k=0}^{\infty }\sum_{j\in W(n_{1}, n_{2})}((A(k\cdot n_{2}-n_{1})+A(k\cdot n_{2}))\otimes A(j))F(A\otimes A),\] 
 \[Q'(G, n_{1}, n_{2})=\sum_{k=0}^{\infty }\sum_{j\in W(n_{1}, n_{2})}(A(j)\otimes (A(k\cdot n_{2})+A(k\cdot n_{2}-n_{1}))G(A\otimes A),\] 
where 
$j\in W(n_{1}, n_{2})$ if and only if  the interval $[j,j+\alpha ]$ is disjoint with all intervals $[kn_{2}-n_{1}, kn_{2}+n_{1}]$ for $k=0, 1,2, \ldots $.
   
$ $

{\bf An outline of the proof now follows.}
\begin{itemize}
\item In Theorem \ref{S} we will find a nil ideal $N$ in $A\otimes A$ and appropriate linear subspaces $S_{i}\subseteq A\otimes A$ and  such that   $N\subseteq \sum_{i=1}^{\infty} S_{i}$.
 The proof of Theorem \ref{S} will be presented  in sections \ref{appendix1}-\ref{appendix4} in the Appendix at the end of this paper.
\item In Theorem \ref{important3} we find  linear spaces $U_{i}\subseteq A(a_{i})$  such that roughly speaking $S_{i}\subseteq A\otimes U_{i} +U_{i}\otimes A$ for $i=1,2, \ldots $.  In Section \ref{lm} linear mappings are introduced, which are important for the construction of sets $U_{i}$.  
\item In Sections \ref{Golod} and  \ref{Hil}  the Golod-Shafarevich Construction and Hilbert Nullstellensatz Theorem are used, to assure that  the codimesion of $U_{i}$ in $A(a_{i})$ is $1$. 
\item In Corollary \ref{N}, using sets $U_{i}$, special linear space $D$ will be constructed such that $N\subseteq D\otimes A+A\otimes D.$
\item In Section \ref{ideal}  we define $I[D]=\{r\in D: Ar\subseteq D\}$. Here it is shown that 
 $I[D]$ is an ideal in $A$. 
\item In Section \ref{final}  it is shown that $A/I[D]\otimes (A/I[D])^{op}$ and $A/I[D]\otimes (A/I[D])$  are nil algebras.

 \end{itemize}
\section{ Embedding  nil ideals in linear spaces}\label{malta}
  The aim of this section is to demonstrate that some  nil ideals in $A\otimes A$ are contained in 
 the sum of appropriate linear spaces $S_{i}.$

Let $x_{i}\in \{x,y\}$ be generators of $A$, then as usual we write $(x_{1}x_{2}\cdots x_{n})^{op}=x_{n}x_{n-1}\cdots x_{1}$.
Let $a_{i}, b_{i}\in A$. For $r=\sum_{i}a_{i}\otimes b_{i}$ define $\xi(r)=\sum_{i}a_{i}\otimes (b_{i})^{op}$.

\begin{theorem}\label{S}  Let $K$ be a field and let $A$ be a $K$-algebra generated in degree one by two elements. 
Let $0<a_{1}< a_{2}< \ldots $ be  integers such that $a_{i}$ divides $a_{i+1}$,  $a_{i+1}>100ia_{i}$, for every $i=1,2, \ldots $, and $a_{1}>100$.
  Then, for all $i\geq 1$  there exist homogeneous linear spaces $G_{i}\subseteq  A(a_{i})\otimes A(a_{i}),$  
 $F _{i}\subseteq A(a_{i})\otimes \sum_{j>{a_{i}\over 10i}}^{10ia_{i}}A(j),$ $ F'_{i}\subseteq   \sum_{j>{a_{i}\over 10i}}^{10ia_{i}}A(j)\otimes A(a_{i})$ with $\dim _{K} G_{i}, \dim _{K} F_{i}, \dim _{K}F'_{i} < 2^{500ia_{i-1}}a_{i+2}^{100i},$ also  
  $G_{1}=0, G_{2}=0,$ $F_{1}=0,$ $F'_{1}=0.$
 Moreover, there exist an ideal $N$ in $A\otimes A$ such that $N\subseteq \sum_{n=1}^{\infty }S_{n}$ where $S_{n}$ is a linear space such that   $S_{n}= \sum _{n=1}^{\infty }M_{n}+Q_{n }+Q_{n-1}$ and  where 
\[M_{n}=\sum_{0\leq k,k'}(A(k a_{n+1})\otimes (A(k' a_{n+1})+A(k' a_{n+1}-a_{n})))G_{n}(A\otimes A)\] 
\[Q_{n}=Q(F_{n}, a_{n}, a_{n+1})+Q'(F'_{n}, a_{n}, a_{n+1}), Q_{0}=0.\]
  Moreover, for every $f\in A\otimes A$ there is $n=n(f)$ such that $f^{n}, \xi (f^{n})\in N$. 
\end{theorem}
\begin{proof} The proof is long and technical,   and in  most cases follows ideas from \cite{s}. The detailed proof  is presented in the Appendix,  sections \ref{appendix1}-\ref{appendix4} at the end of this paper. Observe that $\xi $ is defined at the beginning of this theorem.
\end{proof}

 Notice that $f_{i}$ in Theorem \ref{S} need not be nonzero.
\section{Linear mappings }\label{lm}
In this section we will introduce sets $U_{i}\subseteq A(a_{i})$  and sets $B_{a_{i+1}}(U_{i})$, which will be related to sets $Q_{1},
Q_{2}, Q_{3}$ from Theorem \ref{S}. 
Throughout this paper we will use the following notation.
 Let $n,m$ be natural numbers and  $m<n$. Given set $U\subseteq A(m)$ define \[B_{n}(U)=\sum_{k=0}^{\infty }A(k\cdot n)U_{i}A+A(k\cdot n-m)U_{i}A.\]
For each $i$, denote \[A[a_{i+1}]= \sum_{k=1}^{\infty }A(k\cdot a_{i+1}).\]
 Given linear mapping $f$, by $ker f$ we will denote the kernel of $f$.
 We begin with the following Lemma.
\begin{lemma}\label{U} 
 Let $a_{1}, a_{2}, \ldots $ be a sequence of natural numbers such that $2a_{i}<a_{i+1}$ and  $a_{i}$ divides $a_{i+1}$ for all $i$, and $a_{0}=0$. Suppose that for  $1\leq i\leq n$ we are given a linear space $U_{i}\subseteq A(a_{i})$, and a monomial $m_{i}\in A(a_{i})$   such that \[U_{i}\oplus Km_{i}= A(a_{i}),A(a_{i})\cap \sum_{j=1}^{i-1}B_{a_{j+1}}(U_{j})\subseteq U_{i}.\] 
 Then for $2\leq i\leq n+1$ there are linear mappings $\delta _{i}:A[a_{i}]\rightarrow A$ such that for all $k\geq 0,i\geq 2,$ \[\delta _{i}(A(k\cdot a_{i}))=A(k\cdot a_{i}\cdot t _{i}), A[a_{i}]\cap \sum_{j=1}^{i-1}B_{a_{j+1}}(U_{i})\subseteq ker \delta _{i}\] where $t _{i}=\prod_{j=1}^{i-1}(1-{2a_{j}\over a_{j+1}}).$ 
\end{lemma}
\begin{proof} Induction by $n$. For $n=2$ we define $\delta _{2}$ in the following way.  Let $\gamma _{1}:A(a_{1})\rightarrow K$ be a linear mapping with kernel equal to $U_{1}$ and such that $\gamma _{1}(m_{1})=1$.  Let $\delta  _{2}:A(a_{2})\rightarrow A$  be a linear mapping defined as follows, if $v=v_{1}v_{2}v_{3}$ for monomials  $v_{1}, v_{3}\in M(a_{1})$ and  $v_{2}\in M(a_{2}-2a_{1})$,
  then  $\delta _{2} (v)=\gamma_{1}(v_{1})\cdot v_{2}\cdot \gamma _{1}(v_{3})$. We then extend $\delta _{2}$ by linearity to all elements of $A(a_{2})$. 

 We now extend  mapping $\delta _{2}$ by linearity to all homogeneous elements with degrees divisible by $a_{2}.$  Let $w=\prod_{i=1}^{k} w_{i}$ with all $w_{i}\in M(a_{2})$. We can now define $\delta _{2}(w)=\prod_{i=1}^{k}\delta _{2}(w_{i}),$ and extend $\delta _{2} $ by linearity to all elements from $A[a_{2}].$

 Observe that $\delta _{2}(A(a_{2}))=A(a_{2}-2a_{1})=A(a_{2}(1-{2a_{1}\over a_{2}})),$ so 
$\delta _{2}(A(k\cdot a_{2}))=A(k\cdot a_{2}t_{2}),$ as required.

 Observe that if $p\in B_{a_{2}}(U_{1})$ and $p$ is a homogeneous element of degree divisible by $a_{2}$, then $\delta _{2}(p)=0$. It follows because every $p\in B_{a_{2}}(U_{1})$ is a linear combination of elements from $M(ka_{2})U_{1}M$ and  from $M(ka_{2})M(a_{2}-a_{1})U_{1}M$, images of such elements under $\delta _{2}$ are zero.

  Suppose now that for every $i\leq n$ we defined mappings $\delta _{i}:A[a_{i}]\rightarrow A[a_{i}]$ in a such way that for every $i\leq n$ we have  
\[\delta _{i}(A(k\cdot a_{i}))=A(k\cdot a_{i}\cdot t _{i}), A[a_{i}]\cap \sum_{j=1}^{i-1}B_{a_{j+1}}(U_{i})\subseteq ker \delta _{i}.\]

  To construct  $\delta _{n+1}$ we first  define  $\gamma _{n}:A(a_{n})\rightarrow A(a_{n})$. Let $\gamma _{n}$  be a linear mapping which has  kernel $U_{n}$ and $\gamma _{n}(m_{n+1})=1$.  Let $\delta _{n+1}:A(a_{n+1})\rightarrow A(a_{n+1})$  be a linear map defined as follows, if $v=v_{1}v_{2}v_{3}$ for monomials $v_{1}, v_{3}\in M(a_{n})$ and $v_{2}\in M(a_{n+1}-2a_{n})$,
  then  $\delta _{n+1} (v)=\gamma_{n}(v_{1})\delta_{n}(v_{2})\gamma _{n}(v_{3})$. We then extend $\delta _{n+1}$ by linearity to all elements of $A(a_{n+1})$.  
 
We now extend mapping $\delta _{n+1}$  to elements with  degrees divisible  by $a_{n+1}$ in the following way.  Let $w=\prod_{i=1}^{k} w_{i}$ with all $w_{i}\in A(a_{n+1})$. We can now define $\delta _{n+1}(w)=\prod_{i=1}^{k}\delta _{n+1}(w_{i})$,
 and extend $\delta _{n+1}$ by linearity to all elements from $A[a_{n}].$

 Observe that $\delta _{n+1}(A(a_{n+1}))=\delta _{n}(A(a_{n+1}-2a_{n}))=A((a_{n+1}-2a_{n})t_{n})=A(a_{n+1}(1-{2a_{n}\over a_{n+1}})t_{n})=A(a_{n+1}t_{n+1}),$ as required.
 
We will now show that  if $p\in  \sum_{i=1}^{n}B_{a_{i+1}}(U_{i})$ and $p$ is a homogeneous element of degree divisible by $a_{n+1}$, then $\delta _{n+1}(p)=0$.
 By the inductive  assumption we may assume that if $p'\in  \sum_{i=1}^{n-1}B(U_{i})$ and $p'$ is a homogeneous element of degree divisible by $a_{n}$, then $\delta _{n}(p')=0$.
As $\delta _{n+1}$ is a linear mapping, it suffices to consider the following two cases only.

{\bf Case $1.$} $p\in B$  where $B=\sum_{i=1}^{n-1}B_{a_{i+1}}(U_{i}),$ 
 If $p\in A(k\cdot a_{n+1})$   and  $p\in B,$ then $p$ is a linear combination of elements from $M(ka_{n+1})rM$  where $r\in A(a_{n+1})$  and either $r\in (A(a_{n})\cap B)M(a_{n+1}-a_{n}),$
 or $r\in M(a_{n})(A(a_{n+1}-2a_{n})\cap B)M(a_{n})$ or $r\in M(a_{n+1}-a_{n}) (A(a_{n})\cap B)$ 
 (because $a_{n}$ divides $a_{n+1}$). By Assumptions $B\cap A(a_{n})\subseteq U_{n}\subseteq ker \gamma _{n}.$
  By the definition of the mapping $\delta _{n+1}$,  we get 
 $\delta _{n+1}(r)=0$ in all these cases, because $\delta _{n}(B\cap A(k\cdot a_{n}))=0$ by the inductive assumption.

{\bf Case $2.$} $p\in B_{a_{n+1}}(U_{n})$. Observe now that if $p\in A(k\cdot a_{n+1})$  and  $p\in B_{a_{n+1}}(U_{n})$ then $p$ is a linear combination of elements from $M(ka_{n+1})rM$  where $r\in A(a_{n+1})$ and 
 either $r\in U_{n}M(a_{n+1}-a_{n}),$
 or  $r\in A(a_{n+1}-a_{n})U_{n}$. Recall that $\gamma _{n}(U_{n})=0$.  By the definition of the mapping $\delta _{n+1} $ we get 
$\delta _{n}(r)=0$ in all these cases.
\end{proof}

 Let $m,n,\alpha $ be natural numbers with $m>n$, and let $g\in A(\alpha ).$ Define  
\[M(g,n,m)=\sum_{i:[i,i+\alpha ]\in W}A(i)gA\]
 where $[i,i+\alpha ]\in W$ if and only if the interval $[i,i+\alpha ]$ is disjoint with interval $[km-n,km+n]$, for every $k\geq 0.$
 In Section \ref{Hil} we will use the fact that $M(g,n,m)$  are related to some linear spaces $Q(f,n,m)$, $Q'(f, n, m).$

\begin{lemma}\label{E(f)} Let $a_{1}, a_{2}, \ldots $ be a sequence of natural numbers such that $0<2000j^{2}a_{j}<a_{j+1}$ and  $10(j+1)^{2}a_{j}$ divides $a_{j+1}$ for all $j\geq 1,$ and $10$ divides $a_{1}.$ 
 Let $i$ be a natural number. Suppose that for all $1\leq j\leq i$ we are given linear spaces  $U_{j}\subseteq A(a_{j})$ and monomials $m_{j}\in A(a_{j})$  such that $U_{j}\oplus Km_{j}= A(a_{j})$ and 
$A(a_{j})\cap \sum_{k=1}^{j-1}B_{a_{k+1}}(U_{k})\subseteq U_{j}.$ 

Let $n\leq i$ and let $f\in A(\alpha )$ for some $ {a_{n}\over 10n}\leq \alpha \leq 10n\cdot a_{n}$. 
    Denote $F=M(f,a_{n}, a_{n+1})$. 
Suppose that $F\cap A(a_{n+1+k})\subseteq U_{n+1+k},$ for all $k\geq 0,$ with $n+k+1\leq i.$ Let  mappings $\delta _{j}$ be constructed as in the proof of Lemma \ref{U}.
 Then there is a set  $E(f)$ with cardinality smaller than $2a_{n}^{4}\cdot 2^{2a_{n-1}},$ and such that  
$\delta _{n+1+k}(F\cap A[a_{n+1+k}])\subseteq  <E(f)>,$ for all $k\geq 0$ with  $k+n+1\leq i+1$    
where $<E(f)>$ is the ideal generated in $A$ by elements from $E(f)$. Moreover, $E(f)$ is a homogeneous set and 
 $E(f)\subseteq \sum_{i>{a_{n} \over 40n}}^{20na_{n}}A(i)$.
\end{lemma} 
\begin{proof} Recall that $f\in A(\alpha ).$  Denote first $H(f)=\{\delta _{n}(M(i)fM(j)): i,j< a_{n}, $ and $a_{n}$ divides $i+\alpha +j\}  $.
 Then $H(f)\subseteq \delta _{n}(\sum_{i={a_{n}\over 10n}}^{10na_{n}}A(i))\subseteq \sum_{i={t_{n}a_{n}\over 10n}}^{10nt_{n}a_{n}}A(i).$ We will first show that 
$\delta _{n+1+k}(F\cap A[a_{n+1+k}])\subseteq  <H(f)>,$ for all $k\geq 0$ with  $k+n+1\leq i+1$    
where $<H(f)>$ is the ideal generated in $A$ by elements from $H(f)$.

We will first show that $\delta _{n+1}(F\cap A[a_{n+1}])\subseteq  <H(f)>.$ Recall that elements from $F\cap A[a_{n+1}]$ consist of linear combination of elements from $M(j)fM(j')$ where $j+\alpha +j'=\beta \cdot a_{n+1}$, for some $\beta $
 and the interval $[j,j+\alpha ]$ is disjoint with any interval $[ka_{n+1}-a_{n}, ka_{n+1}+a_{n}.]$
  We can write $j=c+c'$, $j'=d+d'$ where $c',d'$ are divisible by $a_{n+1}$ and $0\leq c,d<a_{n+1}$. As $\alpha \leq {1\over 200}a_{n+1}$ and $c+\alpha +d$ is divisible by $a_{n+1}$ it follows that 
 either  $c+\alpha +d=a_{n+1}$ or  $c+\alpha +d=2a_{n+1}$. If the later holds then $c,d\geq a_{n+1}-\alpha $ which implies that the interval $[c,c+\alpha ]$ is not disjoint with interval $[a_{n+1}-a_{n}, a_{n+1}+a_{n}],$ 
  hence the interval $[j,j+\alpha ]$ is not disjoint with some interval $[ka_{n+1}-a_{n}, ka_{n+1}+a_{n}.]$
 It follows that  $c+\alpha +d=a_{n+1}.$ 
 Observe now that, by the definition of $\delta _{n+1},$ and by Lemma \ref{U} we get
 $\delta _{n+1}(M(j)fM(j'))=\delta _{n+1} (M(c'))\delta _{n+1} (M(c)fM(d))\delta _{n+1}(M(d'))\subseteq 
A\delta _{n+1} (M(c)fM(d))A.$ Consequently in is sufficient to show that 
 $\delta _{n+1} (M(c)fM(d))\subseteq  <H(f)>.$

 Recall that $c+\alpha +d=a_{n+1}$ and  the interval
 $[c,c+\alpha ]$ is disjoint with interval  $[a_{n+1}-a_{n}, a_{n+1}+a_{n}]$, so  $\deg d>a_{n}$.  
 Observe also that  $\deg c>a_{n}$ because  the interval
 $[c,c+\alpha ]$ is disjoint with interval  $[-a_{n}, a_{n}].$ We can write $c=c_{1}c_{2}$, $d=d_{1}d_{2}$ where $c_{1}, d_{2}\in M(a_{n}).$
  By the definition of $\sigma _{n+1}$ we get $\delta _{n+1}(cfd)=\delta_{n+1}(c_{1}c_{2}fd_{1}d_{2})=\gamma _{n}(c_{1})\delta _{n}(c_{2}fd_{1}) \gamma _{n}(d_{2})\in K\cdot \delta _{n}(c_{2}fd_{1}).$

 We can write $c_{2}=c_{3}+c_{4}$, $d_{1}=d_{3}+d_{4}$ where $c_{3}, d_{4}$ are  divisible by $a_{n}$, and $c_{4}, d_{3}$ are monomials of degres less than $a_{n}$. 
 By  the definition of $\delta _{n}$ we then have $\delta _{n}(M(c_{2})fM(d_{1}))=\delta _{n}(M(c_{3}))\delta _{n}(M(c_{4})fM(d_{3}))\delta (M(d_{4}))\subseteq A(c_{3}t _{n})\delta _{n}(M(c_{4})fM(d_{3}))A(d_{4}t _{n}).$
 By the definition of $H(f)$ at the beginning of this proof $\delta _{n}(M(c_{4})fM(d_{3}))\in H(f),$ as required.

 We will now show that $\delta _{n+1+k}(F\cap A[a_{n+1+k}])\subseteq  <H(f)>,$ for all $k\geq 0$, by induction on $k.$ If $k=0$ the result follows from the first part of this proof. Suppose that $k\geq 1$ and that it is true for all numbers smaller than $k$.  
 By the inductive assumption $\delta _{n+k}(F\cap A[a_{n+k}])\subseteq  <H(f)>.$ We need to show that $\delta _{n+k+1}(F\cap A[a_{n+k+1}])\subseteq  <H(f)>.$
Recall that elements from $F\cap A[a_{n+1}]$ consist of linear combination of elements from $M(j)fM(j')$ where $j+\alpha +j'=\beta \cdot a_{n+k+1}$, for some $\beta $
 and the interval $[j,j+\alpha ]$ is disjoint with any interval $[la_{n+1}-a_{n}, la_{n+1}+a_{n}]$ (and hence disjoint with any interval $[la_{n+k}-a_{n}, la_{n+k}+a_{n}],$ and any interval $[la_{n+k+1}-a_{n}, la_{n+k+1}+a_{n}] .)$

Similarly as before it suffices to consider the case when $j,j'< a_{n+k+1}$ and either $j+j'+\alpha=a_{n+k+1}$ or $j+j'+\alpha=2a_{n+k+1}$. 
 If the later holds then, $j,j'\geq a_{n+k+1}-\alpha .$  Therefore $[j,j+\alpha ]$ contains point $a_{n+k+1}$ and $j<a_{n+k+1}.$ 
 Therefore, the interval $[j,j+\alpha ]$ is not disjoint with interval $[a_{n+k+1}-a_{n}, a_{n+k+1}+a_{n}],$ a contradiction.
 It follows that $j+j'+\alpha=a_{n+k+1}.$ 
 Observe that interval $[j, j+\alpha ]$ cannot contain any point $l\cdot a_{n+k}$, because then it would not be disjoint with some interval $[l\cdot a_{n+k}-a_{n}, l\cdot a_{n+k}+a_{n}].$
   It follows that either $j+\alpha < a_{n+k}$; or $j>a_{n+k}$ and $j+\alpha <a_{n+k+1}-a_{n+k}$; or $j> a_{n+k+1}-a_{n+k}.$

By the definition of $\delta _{n+k+1},$ if $j+\alpha < a_{n+k}$  then  $\delta _{n+k+1}(M(j)fM(j'))=\gamma _{n+k}(M(j)fM(a_{n+k}-j-\alpha))\delta _{n+k}(M(a_{n+k+1}-2a_{n+k}))\gamma _{n+k}(M(a_{n+k})).$
 As $M(j)fM(a_{n+k}-j-\alpha) \subseteq F\cap A(a_{n+k})\subseteq U_{n+k}\subseteq ker \gamma _{n+k}$ by assumption, so $\delta _{n+k+1}(M(j)fM(j'))=0.$

If  $j>a_{n+k}$ and $j+\alpha <a_{n+k+1}-a_{n+k}$ then $\delta _{n+k+1}(M(j)fM(j'))=\gamma _{n+k}(M(a_{n+k}))\delta _{n+k}(M(j-a_{n+k})fM(j'-a_{n+k}))\gamma _{n+k}(M(a_{n+k})).$
 By the inductive assumption $\delta _{n+k}(M(j-a_{n+k})fM(j'-a_{n+k}))\subseteq <H(f)>$, hence we get  $\delta _{n+k+1}(M(j)fM(j'))\subseteq <H(f)>$.

 If $j>a_{n+k+1}-a_{n+k}$ then  $\alpha +  j'< a_{n+k}$. Hence, $\delta _{n+k+1}(M(j)fM(j'))=\gamma _{n+k}(M(a_{n+k}))\delta _{n+k}(M(a_{n+k+1}-2a_{n+k}))\gamma _{n+k}(M(j-a_{n+k+1}+a_{n+k})fM(j'))=0,$
 since $M(j-a_{n+k+1}+a_{n+k})fM(j')\subseteq F\cap A(a_{n+k})\subseteq U_{n+k}\subseteq \ker \gamma _{n+k}.$

 We have shown that $\delta _{n+1+k}(F\cap A[a_{n+1+k}])\subseteq  <H(f)>,$ for all $k\geq 0$ with  $k+n+1\leq i+1.$

Assume that $c$ is a natural number divisible by $a_{n-1}$, and that $c$ is smaller than $a_{n}$.  Define $h_{c}:A(c)\rightarrow A$ be a function such that: if $c=a_{n-1}$ then $h_{c}=\gamma_{n-1}$;
 if $a_{n-1}<c<a_{n}$ then for a monomial $v=v_{1}v_{2}$ with $v_{1}\in M(c-a_{n-1})$, $v_{2}\in M(a_{n-1})$ we define $h_{c}(v_{1}v_{2})=\delta _{n-1}(v_{1})\gamma _{n-1}(v_{2}),$ we then extend $h_{c}$ by linearity to all elements of $A(c)$ (where $\gamma _{n-1}$  is as in the definition of the mapping $\delta _{n}$). 
 Similarly, define for such $c$, function $h'_{c}:A(c)\rightarrow A$ in the following way:  if $c=a_{n-1}$ then $h_{c}=\gamma_{n-1}$;
 if $a_{n-1}<c<a_{n}$ then for a monomial $v=v_{2}v_{1}$ with $v_{1}\in M(c-a_{n-1})$, $v_{2}\in M(a_{n-1})$ we define $h_{c}(v_{1}v_{2})=\gamma _{n-1}(v_{2})\delta _{n-1}(v_{1}),$ we then extend $h_{c}$ by linearity to all elements of $A(c).$  We define $h(1)=h'(1)=1$ if $r=1.$
 
 Let now $p$ be a natural number divisible by $a_{n}$ and let $0\leq c,d< a_{n},$  $p-c-d$ is divisible by $a_{n}$ (we will  set $h_{p,c,d}=0$ if  $p-c-d$ is not divisible by $a_{n}$). Define function $g_{p,c,d}:A(p)\rightarrow A$ as follows first for monomials:
if $v_{1}\in M(c), v_{2}\in M(p-c-d),$ $v_{3}\in A(d)$ we define $h_{p,c,d}=h_{c}(v_{1})\delta _{n}(v_{2})h'_{d}(v_{3}),$ we then extend $h_{p,c,d}$ by linearity to all elements from $A(p).$

 Let $X=\{M(i)fM(j): 0\leq i,j< a_{n-1}, $ and $a_{n-1}$ divides $i+\alpha +j\}  $. Observe that $X\subseteq \sum_{i=\alpha }^{\alpha +2a_{n-1}-2}A(i),$ and that the cardinality of $X$ is less than $2^{2a_{n-1}}.$
Define \[E(f)= \{ h_{p,c,d}(r): r\in X , p\in [\alpha, \alpha+2a_{n-1}-2], 0\leq c,d<a_{n-1}\}\]

Observe that the set $H(f)$ from the beginning of our proof is spanned by elements of the form $vfv'$ where $v,v'$ are monomials of degree less than $a_{n}$. Write $v=wu$, $v'=u'w'$ where $u,u'$ are monomials of degree $<a_{n-1}$ and $w,w'$ are monomials of degrees $t,t'$ with $t,t'$ divisible by $a_{n-1}.$  
Then by the definition of function $\delta _{n} $ we get  $\delta _{n}(vfv')\in h'_{t}(w)\cdot h_{p,c,d}(ufu')h_{t'}(w')$ where $c=a_{n}-t$ if $t>0$ and $c=0$ if $t=0$; and $c'=a_{n}-t'$ if $t'>0$ and $c'=0$ if $t'=0$ (because $vfv'$ has degree divisible by $a_{n}$ since $vfv'\in H(f)$). It follows that $H(f)\subseteq <E(f)>$.

 Observe that the cardinality of $E(f)$ depends on the number of choices of numbers  $c,d$ of degrees less than $a_{n}$  and on the number of choices of $p$, and of the cardinality of $X.$ Since $p\in [\alpha, \alpha +2a_{n-1}-2]$ we get that  the cardinality of $E(f)$ is smaller than $a_{n}^{2}\cdot 2a_{n-1}\cdot 2^{2a_{n-1}}<2a_{n}^{4}2^{2a_{n-1}}.$

  Recall that $\delta _{n}(A(a_{n}))= A(t_{n}\cdot a_{n}).$ Observe that if $p-c-d\geq 0$ then $p-c-d\geq a_{n}$ and so $h_{p,c,d}(r)\in \sum_{i=a_{n}t_{n}}^{\alpha +2a_{n-1}}A(i),$ for all $r\in X.$ 
 Recall that $\delta _{n}(A(k\cdot a_{n-1})= A(k\cdot t_{n-1}\cdot a_{n-1}).$ If $p-c-d=0$ then   $h_{p,c,d}(r)\in \sum_{i=(\alpha -2a_{n-1})t_{n-1}}^{\alpha +2a_{n-1}}A(i),$ for all $r\in X.$
 Recall that senence $a_{1}, a_{2}, \ldots $ grows very guickly. It can be inductively proved that $t_{n}>{1\over 2}+{1\over 3n}.$  Recall that $\alpha >{a_{n}\over 10n}.$
  It follows that  $E(f)$ is a homogeneous set and 
$E(f)\subseteq \sum_{i>{a_{n} \over 40n}}^{20na_{n}}A(i)$.
\end{proof}

\section{A short section on the Golod-Shafarevich theorem.}\label{Golod}

 We begin by recalling Theorem $2.1$ from \cite{rr} for the special case when $d=2.$ In our situation $A=T=K[x_{1},x_{2}]$, we use algebra $A$ with two generators $x,y,$ so $d=2.$
 
\begin{theorem}\label{r}[Theorem $2.1$, \cite{rr}] Let $I\subseteq A$ be a homogeneous two-sided ideal $I=\oplus _{n=1,2,\ldots }I(n)$.
 Write $A(n)=I(n)\oplus B(n)$ where $b_{n}=\dim _{K} B(n).$ Suppose that ideal $I$ is generated by homogeneous elements  
 $f_{1}, f_{2}, \ldots \in A$. Let $r_{l}$ be the number of $f_{j}$'s  of degree $l$ (thus $r_{1}=0$). Then for all $n\geq 2$
\[b_{n}\geq db_{n-1}-\sum_{j=0}^{n-2}r_{n-j}b_{j}.\]

\end{theorem}
  
As mentioned in \cite{rr}, at the heart of the Golod-Shafarevich construction is the following proposition.

\begin{proposition}\label{golod}\cite{golod}  Let $b_{n}, r_{l}$ be the sequences in Theorem \ref{r}.
 Let $e>0$ be such that $e<{1\over 2}$. If  $r_{l}\leq e^{2}(2-2e)^{l-2}$ for all $l\geq 2$, then $A/I$ is an infinitely dimensional algebra.
\end{proposition}
 A simple proof of this proposition can be found in \cite{rr}, proof of Proposition $3.1.$ 

 We will now apply these lemmas to prove Lemma \ref{as}.

 Recall that 
$t_{i}=\prod_{j=1}^{i-1}(1-{2a_{j}\over a_{j+1}}).$
\begin{lemma}\label{as}  There are natural numbers $0<a_{1}<a_{2}<\ldots $ such that  $200(i+1)^{2}a_{i}$ divides $a_{i+1}$ for every $i\geq 1,$
 which satisfy the following condition:  If $I$ is an ideal generated by  $2^{510na_{n-1}}a_{n+2}^{110n}$ with degrees from $\sum_{i>{a_{n} \over 40n}}^{20na_{n}}A(i)$ then $A/I$ is an infinitely dimensional algebra.
\end{lemma} 
\begin{proof} It can be inductively proved that $t_{i}>{1\over 2}+ {1\over 3i},$ for all $i>1.$ 
Let $a_{1}>1000000^{{1000}^{1000}}$ and define inductively  \[a_{i+1}= a_{i}\cdot (1000(i+1))^{110}.\]

 Let $r_{l}$ denote the number of homogeneous generating relations of degree $l$ in $I$.
So if ${a_{n}\over 40n}\leq l\leq 20na_{n} $ then $r_{l}<2^{510na_{n-1}}a_{n+2}^{110n}$  
 Denote $e={2-2^{1\over 4}\over 2},$ then $2-2e=2^{1\over 4}.$ Denote $\xi=2^{1\over 8},$ so \[\xi ^{2}=e.\]
 Consequently, to show that $r_{l}<e^{2}(2-2e)^{l}$ for every $l$, it suffices to show that  for every $n\geq 1$
    \[2^{510na_{n-1}}a_{n+2}^{110n}<e^{2}(2-e)^{{a_{n}\over 40n}-2}.\] 

 We will first show that 
$2^{510na_{n-1}}< e^{2}\xi^{{a_n\over 40n}-4}.$
 Recall that $a_{n}=a_{n-1}(1000n)^{100}$, so $ \xi^{{a_n\over 40n}}=2^{{1\over 8}\cdot{1\over 40n}\cdot (a_{n-1}\cdot (1000n)^{100})}.$ 
 Notice that $e^{2}>\xi ^{-80}.$
 It follows that $e^{2}\xi^{{a_n\over 40n}-4}>2^{510na_{n-1}}.$

  We will now show that $a_{n+2}^{110n}<\xi ^{{a_n\over 40n}}.$
 Observe first that $\xi ^{{a_n\over 40n}}=2^{{a_n\over 320n}}$ and $a_{n+2}=a_{n}\cdot (1000(n+1))^{100}\cdot (1000(n+2))^{100}.$  
 Observe that $a_{n+2}<a_{n}^{2},$ for $n=1,2, 3, 4,5 ,6,7 $ it follows since $a_{1}$ is very large, and for bigger $n$ it follows since 
$a_{n}=a_{1}\cdot (\prod_{i=2}^{n}(1000i))^{1000}>a_{1}\cdot (1000n)^{100\cdot {n\over 2}-3}.$
 So it suffices to show that ${a_{n}}^{220n}<2^{{a_n\over 320n}},$ or equivalently that $\sqrt {a_{n}}<2^{{a_n\over 320n}\cdot {1\over 440n}},$ 
 Observe now that ${a_{n}\over 320n}\cdot {1\over 440n}>\sqrt {a_{n}}.$ Therefore it suffices to show that  
 $c<2^{c}$ where $c=\sqrt {a_{n}},$ and ithis is true for all $c\geq 1$.

 Hence we proved that  $a_{n+2}^{110n}<\xi ^{{a_n\over 40n}}.$

 As $\xi ^{2}=2-e$ it follows that  $2^{510na_{n-1}}a_{n+2}^{110n}<e^{2}\xi^{{a_n\over 40n}-4}\cdot \xi ^{{a_n\over 40n}}=e^{2}(2-e)^{{a_{n}\over 40n}-2}$ for every $n\geq 1.$
 It follows that $I$ is generated by $r_{m}<  e^{2}\xi^{m-2}$ homogeneous relations of degree $m$, for each $m>1.$
 By Proposition \ref{golod}, $A/I$ is an infinitely dimensional algebra.
\end{proof}

 We will now introduce a simple lemma.

\begin{lemma}\label{mine} Suppose that $r_{n}$ satisfy assumptions of Proposition \ref{golod} for each $i>1.$
Let $J$ be an arbitrary ideal generated by $s_{n}$ homogeneous  elements of degree $n$, with $s_{n}\leq r_{n}$  for each $n>1$. Then for every $m$ 
 \[dim _{K} A(m)/ J(m)>r_{m}-s_{m}.\] 
\end{lemma}
 \begin{proof} Observe first that by Proposition \ref{golod}, $A/I$ is an infinitely dimensional algebra, for every ideal $I$ generated by $r_{i}$ homogeneous elements of degree $i,$ for every $i>1$. 
 
 Suppose on the contrary that $\dim _{K}A(m)/ J(m)\leq r_{m}-s_{m}.$ 
 Observe that there is ideal $I$ generated by at most $r_{i}$ relations of degree $i$ for each $i\geq 2$ and such that 
  $J\subseteq I.$ Moreover, in the degree $m$ ideal $I$ is generated by $s_{m}$ elements from $J$ and by $r_{m}-s_{m}$ elements from $A(m)$ which correspond to the basis of  $A(m)/ J(m)$ (so they are equal to elements of this basis modulo $J(m)$). Then $A(m)/I(m)=0,$
 and so $A/I$ is nilpotent, a contradiction with assumption that $A/I$ is infinitely dimensional. 
\end{proof}
\begin{lemma}\label{GS}
 There exist natural numbers $0<a_{1}<a_{2}<a_{3}< \ldots $ such that  $200(i+1)^{2}a_{i}$ divides $a_{i+1}$ for every $i\geq 1,$
 which satisfy the following condition: if $J$ is an arbitrary ideal in $A$ generated by elements from arbitrary sets $E(Z_{n})\subseteq \sum_{i={a_{n}\over 40n}}^{20na_{n}}A(i)$, with cardinality of $E(Z_{n})$ less than $2^{505na_{n-1}}a_{n+2}^{110n}$ for all $n\geq 1,$ 
 then    
\[\dim _{K} (A(t_{i+1}a_{i+1})/A(t_{i+1}a_{i+1})\cap J)>4\cdot (2^{500(i+1)a_{i}}a_{i+3}^{100(i+1)})+2.\]

 In particular,  if $G_{i+1}$ is as in Lemma \ref{important3}
 then \[\dim _{K} (A(t_{i+1}a_{i+1})/A(t_{i+1}a_{i+1})\cap J)>4\cdot \dim _{K} G_{i+1} +2.\]  
\end{lemma}
\begin{proof} Define $J$ to be an ideal generated by elements from sets $E(Z_{n})$ for all $n$. Let $s_{m}$ 
 be the number of generating relations of degree $m$ in $J$. 
Let $I$ be an ideal generated by relations from sets $E(Z_{n})$ and also by 
$4\cdot (2^{500na_{n-1}}a_{n+2}^{100n})+2$ homogeneous relations of degree $t_{n}a_{n}$ for each $n\geq 1$.

 Notice that $(4\cdot (2^{500na_{n-1}}a_{n+2}^{100n})+2)+2^{505na_{n-1}}a_{n+2}^{110n}<2^{510na_{n-1}}a_{n+2}^{110n}.$ 
 Therefore $I$ satisfies assumptions of Lemma \ref{as} (since $ {1\over 2}<t_{n}\leq 1$). By Lemma \ref{as} $A/I$ is infinitely dimensional algebra. Moreover, by the proof of Lemma \ref{as}, $I$ is generated by $r_{m}$ homogeneous relations of degree $m$ for all $m\geq 0$ , with $r_{m}$ satisfying assumptions of Proposition \ref{golod}.

  By Lemma \ref{mine} for this $I$ and $J$  we have $dim _{K} A(m)/ J(m)>r_{m}-s_{m}$.
 For $m=t_{i+1}a_{i+1}$ we get that $\dim _{K} (A(t_{i+1}a_{i+1})/A(t_{i+1}a_{i+1})\cap J>r_{t_{i+1}a_{i+1}}-s_{t_{i+1}a_{i+1}}.$

 By construction of $I$,  $r_{t_{n+1}a_{i+1}}-s_{t_{n+1}a_{i+1}}= 4\cdot (2^{500(i+1)a_{i}}a_{i+3}^{100(i+1)})+2.$
  Therefore, $\dim _{K} (A(t_{i+1}a_{i+1})/A(t_{i+1}a_{i+1})\cap J) > 4\cdot (2^{500(i+1)a_{i}}a_{i+3}^{100(i+1)})+2.$
 Since $ dim _{K} G_{i+1}\leq 2^{500(i+1)a_{i}}a_{i+3}^{100(i+1)}$, we get 
  $\dim _{K} (A(t_{i+1}a_{i+1})/A(t_{i+1}a_{i+1})\cap J) >4\cdot \dim _{K} G_{i+1}+2.$
\end{proof} 

\section{Applying Golod-Shafarevich theorem and Hilbert  Nullstellensatz theorem.}\label{Hil}

 In this section we first give  supporting lemmas, and then prove Theorem \ref{important3}, which will, roughly speaking, show that linear spaces $S_{i}$ from Theorem \ref{S} are contained in $U_{i}\otimes _{K} A+A\otimes _{K} U_{i},$ for suitable linear spaces $U_{i}\subseteq A.$

We begin with the following lemma,  which closely resembles  Lemma $4.3$ \cite{sb}   and has the same proof. 
\begin{lemma}\label{Hilbert} Let $K$ be an algebraically closed field, let $m$ be a natural number, and let $T\subseteq A(m)$  
  and  $Q\subseteq A(m)\otimes _{K} A(m)$ be $K$-linear spaces such that $\dim _{K} T+4\dim_{K} Q\leq \dim _{K} A(m)-2$.

  Then there exists a linear space $F\subseteq A(m)$ such that $T\subseteq F$ and 
 $Q\subseteq F\otimes _{K} A(m)+A(m)\otimes _{K} F$, and $\dim_{K} A(m)=1+\dim_{K} U$. 
\end{lemma}
\begin{proof} The same proof as in Lemma $4.3$ in \cite{sb} can be used.  Notice that the fact that $m=2^{n}$ in Lemma $4.3$ \cite{sb} is not used in the proof, so the proof works for all $m$. Also
 in the proof  variable $Z$ can be put to be zero to make it easier. The main idea of the proof is to use the Hilbert  Nullstellensatz theorem.
\end{proof}

\begin{lemma}\label{Z} 
Let $i$ be a natural number, and $a_{i+1}<a_{i+2}$ be natural numbers such that $200(i+1)^{2}a_{i+1}$ divides $a_{i+2}.$
Let $U_{i+1}, F_{i+1}, F'_{i+1}$ be linear spaces, and let $U_{i+1}\subseteq A(a_{i+1}),$  $m_{i+1}\in M(a_{i+1})$ be such that $U_{i+1}+Km_{i+1}=A(a_{i+1}).$
 Moreover, let $F_{i+1}\subseteq A(a_{i+1})\otimes \sum_{j={a_{i+1}\over 10(i+1)}}^{10(i+1)a_{i+1}}A(j),$ and $F'_{i+1}\subseteq \sum_{j={a_{i+1}\over 10(i+1)}}^{10(i+1)a_{i+1}}A(j)\otimes A(a_{i+1}).$
 Then there are homogeneous linear spaces $Z'_{i+1}, Z''_{i+1}\subseteq \sum_{j={a_{i+1}\over 10(i+1)}}^{10(i+1)a_{i+1}}A(i)$ with $\dim _{k}Z'_{i+1}\leq \dim K F_{i+1},$ $\dim _{k}Z''_{i+1}\leq \dim K F'_{i}$ and such that 
\[Q(F_{i+1}, a_{i+1}, a_{i+2})+U\otimes A\subseteq U\otimes A+ A\otimes M(Z'_{i+1}, a_{i+1}, a_{i+2}) ,\] 
\[Q'(F'_{i+1}, a_{i+1}, a_{i+2})+A\otimes U\subseteq A\otimes U+M(Z''_{i+1}, a_{i+1}, a_{i+2})\otimes A\] 
where $U=\sum_{k=1}^{\infty } (A(ka_{i+2}-a_{i+1})+A(ka_{i+2}))U_{i+1}A.$
\end{lemma}
\begin{proof} Using the fact that $U_{i+1}+Km_{i+1}=A(a_{i+1})$, it follows that for every $g\in F_{i+1}$ there is $z(g)\in \sum_{j={a_{i+1}\over 10(i+1)}}^{10(i+1)a_{i+1}}A(j),$ and $\alpha \in K$ such that either $g\in \alpha m_{i+1}\otimes z(g)+U_{i+1}\otimes A.$ 
  Let $g_{1}, \ldots g_{m}$ be a basis of $F_{i+1}$. Let $Z'_{i+1}$ be the linear space spanned by elements  $z(g_{1}), z(g_{2}), \ldots , z(g_{m}).$
   By the definition of $ Q(F_{i+1}, a_{i+1}, a_{i+2})$ and of $M(Z_{i+1}, a_{i+1}, a_{i+2})$ it follows that  
  $Q(F_{i+1}, a_{i+1}, a_{i+2})+U\otimes A\subseteq U\otimes A+ A\otimes M(Z'_{i+1}, a_{i+1}, a_{i+2})$. The proof that $Q'(F'_{i+1}, a_{i+1}, a_{i+2})+A\otimes U\subseteq M(Z''_{i+1}, a_{i+1}, a_{i+2})\otimes A+A\otimes U$ is similar.
\end{proof}

 We are now ready to prove the main result of this chapter. 
\begin{theorem}\label{important3}
 Suppose that we are given 
natural numbers $0<a_{1}<a_{2}<a_{3}< \ldots $, which satisfy assumptions of Lemma \ref{GS} and such that $200(i+1)^{2}a_{i}$ divides $a_{i+1}$  for every $i$.
  Suppose that for all $i\geq 1$ we are given homogeneous linear spaces $G_{i}\subseteq  A(a_{i})\otimes A(a_{i}),$ 
 $F _{i}\subseteq A(a_{i})\otimes \sum_{j>{a_{i}\over 10i}}^{10ia_{i}}A(j),$ $ F'_{i}\subseteq   \sum_{j>{a_{i}\over 10i}}^{10a_{i}i}A(j)\otimes A(a_{i})$  with $\dim _{K} G_{i}, \dim _{K} F_{i}, \dim _{K}F'_{i} <2^{500ia_{i-1}}a_{i+2}^{100i},$
 $F_{1}=0, G_{1}=0, G_{2}=0$. Then there are linear spaces $U_{i}$ and monomials $m_{i}$ such that for all $i\geq 1$ we have 
\begin{itemize}
\item[1.]  $U_{i}\oplus Km_{i}=A(a_{i})$ and   $A(a_{i})\cap \sum_{j=1}^{i-1}B_{a_{j+1}}(U_{j})\subseteq U_{i},$  
\item[2.] $G(i)\subseteq U_{i}\otimes A(a_{i})+A(a_{i})\otimes U_{i},$
\end{itemize}
Moreover there are homogeneous linear spaces $Z_{i}\subseteq \sum_{i={a_{i}\over 10i}}^{10ia_{i}}A(i)$ such that $\dim _{Z_{i}}\leq \dim _{K} F_{i}+\dim _{K} F'_{i},$ and such that 
\begin{itemize}
\item[3.] $Q(F_{j},a_{j},a_{j+1}), Q'(F'_{j},a_{j}, a_{j+1})\subseteq D_{i}\otimes A+A\otimes D_{i}$ 
 where $D_{i}=\sum_{j=1}^{i}B_{a_{j+1}}( U_{j})+M(Z_{j},a_{j}, a_{j+1}).$
\item[4.] $A(a_{i})\cap D_{i-1}\subseteq U_{i}.$
\end{itemize}
\end{theorem}
\begin{proof} We proceed by induction on $i$. For $i=1$, $F_{1}, G_{1}=0$ so we can take $U_{1}$ to be any linear subspace of $A(a_{1})$ of codimension $1$.
 Then for an appriopriate monomial $m_{1}$ we have $U_{1}+Km_{1}=A(a_{1}).$ Then we can set $Z_{1}=0.$

 Suppose that the result holds for all numbers not exceeding $i$, we will show that it holds for $i+1.$
  Assume that  we already constructed sets $U_{1}, \ldots , U_{i}$, and sets $Z_{1}, \ldots ,Z_{i}$.
 We will construct set $U_{i+1}$ and then set $Z_{i+1}.$
 
Let $V$ be a linear subspace of $A(a_{i+1})$ such that $V\oplus (D_{i}\cap A(a_{i+1}))=A(a_{i+1}).$
 
We will first  show that \[\dim _{K} (D_{i}\cap A(a_{i+1}))< \dim _{K} A(a_{i+1})- 4\dim _{K} G_{i+1}-2.\]
 We will prove an equivalent statement that 
 $\dim _{K} V>4\dim _{K} G_{i+1}+2$.
 
Let $\delta _{i+1}:A(a_{i+1})\rightarrow A(a_{i+1})$ be as in Lemma \ref{U}. Then $\delta _{i+1}(U')=0,$ where $U'=(\sum_{j=1}^{i}B_{a_{j+1}}( U_{j}))\cap A(a_{i+1}).$  Denote $M'=(\sum_{j=1}^{i} M(Z_{j},a_{j}, a_{j+1}))\cap A(a_{i+1}),$
 so $D_{i}\cap A(a_{i+1})=U'+M'$ 
 Then $\delta _{i+1}(V)+ \delta _{i+1}(M')=\delta _{i+1}(A(a_{i+1}))=A(a_{i+1}\cdot t _{i+1})$.  
 It follows that $\dim _{K}V\geq \dim _{K}\delta _{i+1}(V)\geq \dim _{K} A(a_{i+1}\cdot t _{i+1})-\dim _{K}\delta _{i+1}(M'),$ where $t_{i+1}$ is as in Lemma \ref{U}.

 Fix $n\leq i,$ and let $f\in Z_{n},$ and denote $F=M(f, a_{n}, a_{n+1}).$ Observe that $F\subseteq D_{n+k}$ for every $k\geq 0,$ hence $A(a_{n+k+1})\cap F\subseteq A(a_{n+k+1})\cap D_{n+k},$ for all $k\geq 0.$ 
If $n+k+1\leq i,$ then by the inductive assumption   $A(a_{n+1+k})\cap  D_{n+k}\subseteq U_{n+k+1},$ and so $A(a_{n+1+k})\cap  F\subseteq U_{n+k+1}.$ Therefore $f$ satisfies assumptions of Lemma \ref{E(f)}.
Therefore any element $f$ from sets $Z_{n}$ satisfy assumptions of Lemma \ref{E(f)}. Let $f_{1}, f_{2}, \ldots , f_{k'}$ be a basis of $Z_{n}$,  We can apply Lemma \ref{E(f)} several times for  $f=f_{1}$, $f=f_{2}, \ldots , f=f_{k'}.$ 
 Let $E(Z_{n})=\bigcup_{j=1}^{k'}E(f_{j})$ where $E(f_{j}) $ is as in Lemma \ref{E(f)} applied for $f=f_{j}$, then  
 the  cardinality of $E(Z_{n})$ does not exceede $2a_{n}^{4}\cdot 2^{2a_{n-1}}\cdot \dim _{K} Z_{n}<2^{505a_{n-1}}a_{n+2}^{110i}$. Moreover,  if $i+k+1\leq i+1$ then 
\[\delta _{n+k+1}(F\cap A[a_{n+k+1}])\subseteq  <E(Z_{n})>\]
 (where $<E(Z_{n})>$ is the ideal generated in $A$ by elements from $E(Z_{n})$). 
 In particular we have $\delta _{i+1}(F'\cap A[a_{i+1}])\subseteq \sum_{n=1}^{i} <E(Z_{n})>,$ where we denote $F'=\sum_{j=1}^{k'}M(f_{j}, a_{n}, a_{n+1}).$ Such sets $E(Z_{n})$ can be constructed for all $n\leq i.$
 It follows that  $\delta _{i+1}(D_{i}\cap  A[a_{i+1}])\subseteq  <E(Z_{n})>$ (because $\delta _{i+1}(U')=0$ by Lemma \ref{U}).
 Let $J$ be the ideal of $A$ generated by all these sets $E(Z_{n})$ for all $1\leq n\leq i.$


By Lemma \ref{E(f)}, $E(Z_{n})\subseteq \sum_{i>{a_{n} \over 40n}}^{20na_{n}}A(i)$. It follows that $I$ is generated by elements from sets $E(Z_{n})\subseteq \sum_{i>{a_{n} \over 40n}}^{20na_{n}}A(i)$ with 
  the cardinality of $E(Z_{n})$ less than $2^{505a_{n-1}}a_{n+2}^{110i}$ for $n=1,2, \ldots ,i.$
 By Theorem \ref{GS}, \[\dim _{K}(A(t_{i+1}a_{i+1})/A(t_{i+1}a_{i+1})\cap I)>4\dim _{K} G_{i+1}+2.\]
 Therefore, we obtain  $dim _{K} V\geq \dim _{K} A(a_{i+1}\cdot t _{i+1})-\dim _{K}\delta _{i+1}(M')=\dim _{K} A(t_{n+1}a_{i+1})/A(t_{n+1}a_{i+1})\cap J.$
It follows that $\dim _{K} V> 4 \dim _{K}G_{i+1}+2,$ and so $\dim _{K} (D_{i}\cap A(a_{i+1}))<\dim _{K} A(a_{i+1})- 4\dim _{K} G_{i+1}-2.$

  We can now apply Lemma \ref{Hilbert} for $m=i+1$ and for sets $T=A(a_{i+1})\cap D_{i}, Q=G_{i+1}$, and obtain set $F$ as in Lemma \ref{Hilbert}. By Lemma \ref{Hilbert} we get $T\subseteq F$ and $G_{i+1}=Q\subseteq F\otimes A(a_{i+1})+A(a_{i+1})\otimes F,$
 so we can set  $U_{i+1}=F.$

 We have constructed set $U_{i+1}$ satisfying the thesis of our theorem.
We can now put $Z_{i+1}=Z'_{i+1}+Z''_{i+1}$ where $Z'_{i+1}, Z''_{i+1}$ are as in Lemma \ref{Z}.
\end{proof}
\begin{corollary}\label{N}  Let $a_{0}, a_{1}, \ldots $ satisfy assumptions of Lemma \ref{GS} and Theorem \ref{important3}. Let $U_{i}, Z_{i}$ be as in Theorem \ref{important3}, and let $N$ be as in Theorem \ref{S}. Then $N\subseteq D\otimes A+A\otimes D$ 
 where $D=\sum_{i=1}^{\infty }D_{i}$ with $D_{i}=\sum_{j=1}^{i}B_{a_{j+1}}( U_{j})+M(Z_{j},a_{j}, a_{j+1}).$
Moreover,  $A(a_{i})\cap D_{i-1}\subseteq U_{i},$ for all $i\geq 1.$

Moreover $S_{i}\subseteq A\otimes D_{i+3}+D_{i+3}\otimes A$ for every $i\geq 1$, where $S_{i}$ is as in Theorem \ref{S}.
\end{corollary}
\begin{proof}  By Theorem  \ref{important3}, applied for sets $F_{i}, F_{i}', G_{i}$  as in Theorem \ref{S}, we have 
$A(a_{i})\cap D_{i-1}\subseteq U_{i},$ for all $i\geq 1.$ By Theorem \ref{S}, $N\subseteq \sum_{n=1}^{\infty }S_{n},$
 where $S_{n}= \sum _{n=1}^{\infty }M_{n}+Q_{n }+Q_{n-1}$ and  where $M_{n}=\sum_{0\leq k,k'}((A(k a_{n+1})+(A(k a_{n+1}-a_{n}))\otimes (A(k' a_{n+1})+A(k' a_{n+1}-a_{n})))G_{n}(A\otimes A),$ 
\[Q_{n}=Q(F_{n}, a_{n}, a_{n+1})+Q(F'_{n}, a_{n}, a_{n+1}).\]
 By Lemma \ref{Z}, $Q_{n}\subseteq D\otimes A+A\otimes D.$
 By Theorem \ref{important3}, $G_{i}\subseteq U_{i}\otimes A+A\otimes U_{i},$ hence $M_{n}\subseteq D\otimes A+A\otimes D.$
 It follows that $N\subseteq A\otimes D+D\otimes A.$  

 Similarly, to show that $S_{i}\subseteq A\otimes D_{i+3}+D_{i+3}\otimes A$ observe that $S_{i}=M_{n}+Q_{n}+Q_{n-1}$.
 By Theorem \ref{important3} (2), and Lemma \ref{Z}, we have  $M_{n}\subseteq D_{i+3}\otimes A+A\otimes D_{i+3}$
 because $\sum_{j=1}^{i}B_{a_{j+1}}(U_{j})\subseteq D_{i}.$
 By Lemma $6.2$ and by Theorem \ref{important3} (3), $Q_{n}, Q_{n-1}\in D_{n+3}\otimes A+A\otimes D_{n+3}.$
\end{proof}
\begin{corollary}\label{n2}
 Let notation be as in Corollary \ref{N}. Then there are homogeneous ideals $I_{1}, I_{2}, \ldots ,$ in $A\otimes A$ and such that $N=I_{1}+I_{2}+\ldots $, and for every $n,$
$I_{n}\subseteq D_{n}\otimes A+A\otimes D_{n}$. Moreover, for every $f\in A\otimes A$ there is $n=n(f)$ such that $f^{n}, \xi (f^{n})\in N'$ where  $N'=\sum_{i=1}^{\infty }( A^{a_{n+2}}\otimes A^{a_{n+2}})I_{n}$.
\end{corollary}
\begin{proof}
 Recall that $N$ is constructed as in Theorem \ref{S}, hence for every $f\in A\otimes A$ there is $n=n(f)$ such that $f^{n}, \xi (f^{n})\in N$. We put $I_{1}=I_{2}=\ldots =I_{9}=0$. If $K$ is a countable field then for $n\geq 1$, we define   
  $I_{n+9}$ to be the ideal of $A\otimes A$ 
 generated by elements $f_{n}^{20na_{n+1} }$ and $\xi(f_{n}^{20na_{n+1}})$ where $f_{n}$ is as in the beginning of the proof of Theorem \ref{S}. If $K$ is an uncountable field then for $n\geq 1$, we define   for every $n\geq 1,$ 
  $I_{n+9}$ be the ideal of $A\otimes A$ generated by elements from the set $\{f^{100a_{n+2}}:f\in {\bar H}_{n}\}$ and 
  from the set $\{\xi(f)^{100a_{n+2}}:f\in {\bar H}_{n}\}$. From the proof of Theorem \ref{S} it follows that $I_{n+9}\subseteq \sum_{i=1}^{n+5}S_{i}$, where $S_{i}$ are as in Theorem \ref{S}.   From the last part of Corollary \ref{N}, $S_{n}\subseteq D_{n+3}\otimes A+A\otimes D_{n+3}$ for every $n\geq 1$. Since $D_{i}\subseteq D_{i+1}$ for every $n$, it follows that  
 $I_{n}\subseteq D_{n}\otimes A+A\otimes D_{n}$.

We will now show that for every $f\in A\otimes A$ there is $n=n(f)$ such that $f^{n}, \xi (f^{n})\in N'$. Let $f\in N$, then by construction of ideals $I_{1}, I_{2}, \ldots $ there is $m$ such that $f^{n}, \xi (f^{n})\in I_{m}$ for some $n$. Then $f^{n+a_{m+2}}=f^{a_{m+2}}\cdot f ^{n}\subseteq 
 ( A^{a_{m+2}}\otimes A^{a_{m+2}})I_{m}\subseteq N'$.

Observe also that $\xi(f^{n+2a_{m+2}})=\xi(f^{a_{m+2}}\cdot f^{n}\cdot f^{a_{m+2}})\subseteq (A^{a_{m+2}}\otimes A^{a_{m+2}})I_{m}\subseteq N'.$ 
\end{proof}

\section{Definition of the ideal $I$} \label{ideal}
 Let  $C$ be a right ideal of $A$, so $IA\subseteq A$, which is homogeneous.
Define subspace $I[C]$ of $A$ in the following way.
\[\{I[C]=\{r\in C: Ar\subseteq C\}.\]  
 Clearly $I(C)$ is an ideal in $A$.

Let $D=\bigcup_{i=1}^{\infty }D_{i},$ with $D_{i}=\sum_{j=1}^{i}B_{a_{j+1}}( U_{j})+M(Z_{j},a_{j}, a_{j+1}),$ be as in Corollary \ref{N}. We can define $I[D]=\{a\in D:Aa\subseteq  D\}.$
 As $D$ is a right ideal, so $DA\subseteq D$, so it is clear that $I[D]$ is an ideal in $D$.

\begin{lemma}\label{d} Let notations be as in Corollary \ref{N} and Theorem \ref{important3}. 
 If $a\in \bigcap _{i=0}^{a_{n+1}}A(i)D_{n}$ then $a\in I[D]$.
\end{lemma}
\begin{proof} By the definition of $D_{n}$, $A(a_{n+1})D_{n}\subseteq D_{n}$. Let $v\in M$, we will show that $va\in D_{n},$ and hence $va\in D$.  
 Let  $v=v_{1}v_{2}$ where the degree of $v_{1}$  divisible by $a_{n+1},$ $v_{2}\in M(s)$ with $s<a_{n+1}.$
  Notice that by assumption $a\in A(a_{n+1}-s)D_{n}$ hence $v_{2}a\in v_{2}A(a_{n+1}-s)D_{n}\subseteq A(n+1)D_{n}\subseteq D_{n}$. Therefore $v_{1}v_{2}a\in v_{1}D_{n}\subseteq D_{n}\subseteq D.$
 
So for every $v\in M$, $va\in D$, hence $Aa\subseteq D$, so $a\in I[D]$. 
 \end{proof}
\begin{lemma}
 Let notations be as in Corollary \ref{N} and Theorem \ref{important3}. 
  Then $A/I[D]$ is an  infinitely dimensional space.
\end{lemma}
\begin{proof} As $I(D)\subseteq D$, it suffices to show that $A/D$ is an infinitely dimensional linear space. 
 By the definition $D=\sum_{j=1}^{\infty }C_{j}$, where for every $j$, $C_{j}=B_{a_{j+1}}( U_{j})+M(Z_{j},a_{j}, a_{j+1})$  is a homogeneous linear space and all elements from $C_{j}$ have degrees 
  ${a_{j}\over 40j}$ or larger (elements with the smallest possible degrees comming from sets $Z_{j}$ or $U_{j}.$)
 Suppose on the contrary that $A/D$ is finitely dimensional, then it is nilpotent of degree $j$, for some $j$. Then, $A(j)\subseteq A(j)\cap \sum_{i=1}^{j}C_{i}$, so $A(j)\subseteq A\cap D_{j}.$
 As $D_{j}$ is a right ideal it follows that  $A(j+k)\subseteq D_{j}A(k)\subseteq D_{j}$ for all $k\geq 0$.
 Therefore,  $A(a_{j+k})\cap D_{j}=A(j+k)$ for all  $k\geq 0$. Observe that for $k=2$,  $D_{j}\subseteq D_{j+1}$ and by Corollary \ref{N} we get  
 $A(a_{j+2})\cap D_{j}\subseteq A(a_{j+2})\cap D_{j+1}\subseteq U_{j+2}\neq A(a_{j+2}),$ a contradiction.
\end{proof}

We will now prove two supporting Lemmas.

\begin{lemma}\label{q}
 Let $Q_{1}, Q_{2}, B\subseteq A$, are linear spaces. If $a\in A\otimes Q_{1}+B\otimes A$ and $a\in A\otimes Q_{2}+B\otimes A$ then $a\in A\otimes (Q_{1}\cap Q_{2})+B\otimes A.$
 Similarly, if $a'\in  Q_{1}\otimes A+ A\otimes B$ and $a'\in Q_{2}\otimes A+ A\otimes B$, then $a'\in (Q_{1}\cap Q_{2})\otimes A+A\otimes B.$
\end{lemma}
\begin{proof} If we take $B'=A/B$ then $a+B\in B'\otimes Q_{1}, a+B\in B'\otimes Q_{2},$ hence $a+B\in B'\otimes (Q_{1}\cap Q_{2}),$ hence  $a\in A\otimes (Q_{1}\cap Q_{2})+B\otimes A.$
 The proof of the second statement is similar.
\end{proof}

\begin{lemma}\label{2} Let $n$ be a natural number, and 
 let $Q_{1}, \ldots , Q_{n}$ be linear subspaces of $A$, and let $a\in A\otimes A$ be such that 
$a\in Q_{i}\otimes A+A\otimes Q_{j}$ for all $i,j\leq n$. Then $a\in (\bigcap _{i=1}^{n}Q_{i})\otimes A+A\otimes  (\bigcap _{i=1}^{n}Q_{i})$.
\end{lemma}
\begin{proof} Observe that by the first part of Lemma \ref{q} applied several times (for sets $Q_{1}, \ldots Q_{n}$) we get, that for every $j\leq n$, $a\in A\otimes \bigcap   _{i=1}^{n}Q_{i}+Q_{j}\otimes A$. 
  By the second part of Lemma \ref{q} applied to $a'=a$ we get $a\in  (\bigcap _{i=1}^{n}Q_{i})\otimes A+A\otimes  (\bigcap _{i=1}^{n}Q_{i})$.
\end{proof}

\begin{lemma}\label{4}  Let notation be as in Corollary \ref{N} and Theorem \ref{important3}.  Let $I_{n}$ be a homogeneous ideal of $A\otimes A$ such 
that $I_{n}\subseteq  A\otimes D_{n}+D_{n}\otimes A$ for some $n$. Then $(A^{a_{n+2}}\otimes A^{a_{n+2}}) I_{n}\subseteq A\otimes I[D]+I[D]\otimes A.$
\end{lemma}
\begin{proof} We will first show that $(A^{a_{n+2}}\otimes A^{a_{n+2}})I_{n}$ is contained in  $A(k)D_{n}\otimes A+A\otimes A(j)D_{n}$ for each $0\leq i,j\leq a_{n+1}$.
 Note that $(A^{a_{n+2}}\otimes A^{a_{n+2}})I_{n}=(A(k)\otimes A(j))(A^{a_{n+2}-k}\otimes A^{a_{n+2}-j}) I_{n}\subseteq (A(k)\otimes A(j))I_{n}\subseteq (A(k)\otimes A(j))(D_{n}\otimes A+A\otimes D_{n})\subseteq 
A(k)D_{n}\otimes A+A\otimes A(j)D_{n}.$

Let $a\in A^{a_{n+2}}I_{n}$. Then $a$ satisfies assumptions of Lemma \ref{2} for $Q_{i}=A(i)D_{n}$. By Lemma \ref{2}, $a\in A\otimes P=P\otimes A$ where $P=\bigcap _{i=0}^{a_{n+1}} A(i)D_{n}.$
 By Lemma \ref{d}, $P\subseteq I[D]$. Hence $a\in D_{n}\otimes A+A\otimes A(j)D_{n}.$
\end{proof}


\section{Proof of Theorem \ref{1}}\label{final}

We are now ready to prove Theorem \ref{1}.
 
 {\bf Proof of Theorem \ref{1}} 
 Let notation be as in Corollary \ref{N},  Corollary \ref{n2} and Theorem \ref{S}. Let $D=\sum_{i=1}^{\infty} D_{i}$ as in Corollary \ref{N} and let $I[D]=\{r\in D: Ar\subseteq D\}$.  Let $R=A/I[D]$, we claim that $R\otimes R$ and $R\otimes R^{op}$ are nil algebras.

By Corollary \ref{n2}  there are homogeneous ideals $I_{1}, I_{2}, \ldots ,$ in $A\otimes A$ and such that $N=I_{1}+I_{2}+\ldots $, and for every $n,$
$I_{n}\subseteq D_{n}\otimes A+A\otimes D_{n}$. 
  Moreover, for every $f\in A\otimes A$ there is $n=n(f)$ such that $f^{n}, \xi (f^{n})\in N'$, where  $N'=\sum_{i=1}^{\infty }( A^{a_{n+2}}\otimes A^{a_{n+2}})I_{n}$.
 Observe that by Lemma \ref{4},  $(A^{a_{n+2}}\otimes A^{a_{n+2}}) I_{n}\subseteq A\otimes I[D]+I[D]\otimes A.$ 
 Therefore $N'\subseteq A\otimes I[D]+I[D]\otimes A.$
 By Assumption for every $f\in A\otimes A$ there is number $n=n(f)$ such that $f^{n}\in N' \subseteq A\otimes I[D]+I[D]\otimes A.$
 It follows that $(A/I)\otimes _{K} (A/I)=R\otimes _{K} R$ is a nil algebra.

Let $g\in A\otimes A$, and let $f\in A\otimes A$ be such that $\xi(f)=g$ in the linear space $A\otimes A$. By assumption there is number $n=n(f)$ such that $\xi(f^{n})\in N' \subseteq A\otimes I[D]+I[D]\otimes A.$  

Let $*$ be the opposite multiplication. Let $g'=g*g*\ldots *g$ -the opposite multiplication of $m$ copies of $g$. Then in the linear space $A\otimes A$ we have $g'=\xi(f^{m})$. By Assumption $\xi(f^{n})\in N'$ hence $g'\subseteq J$ where $J= A\otimes I[D]+I[D]\otimes A.$
 $J$ is a linear subspace in $A\otimes A$, so $g'+J=0$ in $A/I[D]\otimes A/I[D].$
It follows that $(g+J)*(g+J)* \ldots *(g+J)\subseteq g'+J$, so it follows that $g+J$ is a nilpotent element in $A/I[D]\otimes (A/I[D])^{op}$ for all $g\in A$.
 Therefore, $R\otimes R^{op}$ is a nil algebra.

\section{Appendix: Proof of Theorem \ref{S}.\\
 Nil ideals}\label{appendix1}
 In this section we will recall some ideas from \cite{s}. Results from \cite{s} were simplified in \cite{n}, so we will use  \cite{n}.
 We first introduce the following notation: 
 let  $r\in ( A(1)+\ldots +A(t))\otimes (A(1))+\ldots +A(q))$, where $t, q$ are the minimal possible. Then we say that the left degree of $a$, denoted $l(a)$  is $t$, and the right degree of $a$ denoted $r(a)$ is $q$. 

We begin with three lemmas which correspond to Lemma 1 from \cite{s} and Lemma 7 from \cite{n}.
 
\begin{lemma}\label{left1}
  Let $ f \in \sum_{1\leq i,j\leq t}A(i)\otimes A(j)$ for some $t \geq 1$. For each $n \geq  t$ there exists a linear space  
 $F(f,n)\subseteq A(n)\otimes A$   with
 $$ \dim _{K} F(f,n) < 2^{2t+1}tn$$ such that, for every integer $ m \geq n$,  $f^{m}$ belongs to the right ideal of $A\otimes A$ generated by polynomials of the form $hgf^{m-n}$, where $h\in F(f,n)$ and $g\in (A(0)+\ldots +A(t))\otimes K$.
\end{lemma}
\begin{proof} We will modify, in a straightforward way, the proof of Lemma $7$ in \cite{n}, when we only consider the first component in $A\otimes A$.
  We write $f=\sum _{i=1}^{t}f_{i}$, where $f_{i}\in A(i)\otimes A$; thus $f^{m}$ is a sum of terms of the form
 $f_{i_{1}}f_{i_{2}}\ldots    f_{i_{k}}f_{i_{k+1}}\ldots  f_{i_{m}}$, where k denotes the largest subscript such that $i_{1} +i_{2} +···+i_{k} \leq  n$. We then have the decomposition:
\[f^{m} =\sum_{k\leq n}\sum_{i_{1}+\ldots +i_{k}=n}f_{i_{1}}\ldots f_{i_{k}}f^{m-k}+\sum_{j=1}^{t-1}(\sum_{i_{1}+\ldots +i_{k}=n-j, i_{k+1}>j}f_{i_{1}}\ldots f_{i_{k}}f_{i_{k+1}}f^{m-k-1}
))\]

Hence 
\[f^{n}=\sum_{k\leq n}v_{0,k}f^{m-k}+\sum_{k\leq n-1}(\sum_{j=1}^{t-1}v_{j,k}(f_{j+1}+f_{j+2}+\ldots +f_{t})f^{m-k-1})\]

where $v_{j,k}$ is the homogeneous component of degree $n - j$ in $f^{k}$ with respect to the first component in the tensor product (that means that $f^{k}=\sum_{j}v_{j,k}$ and  $v_{j,k}\in A(n-j)\otimes A$, for each $j$).

 Taking $F(f,n) = \sum_{i=0}^{t}\sum_{j=0}^{t-1} \sum_{k=1}^{n}v_{j,k}(M(j)\otimes A(i))$, the representation above has the desired properties (where $K$ is the base field). To see this, note that in the second summation each term $v_{j,k}f_{j+i}f^{m-k-1}$ lies in the span of all members of  $v_{j,k}(M(j+i)\otimes \sum_{j'=1}^{t}M(j'))(M(i))f^{m-n}f^{n-k-1}\subseteq  v_{j,k}(M(j)\otimes \sum_{j'=1}^{t}M(j'))(M(i)\otimes K)f^{m-n}(A\otimes A)$.  As elements of $v_{j,k}(M(j)\otimes \sum_{j'=1}^{t}M(j'))$ lie in $F(f,n)$, and elements of $M(i)$ have degree $< t$, the product resides in the desired right ideal. The same is still more easily seen for terms of the first summation. Notice that if we want $v_{j,k}$ to be nonzero, then we necessarily must have ${{n-j}\over t} \leq k \leq  n-j$. We compute
\[\dim _{K} G(f,n)<\sum_{j=0}^{t-1}\sum_{k\geq {(n-j)\over t}}^{n-j}2^{t}2^{t+1}\leq 2^{2t+1}tn\]
\end{proof}

 Let $i,j$ be nonnegative integers. For $a\in A(i)\otimes  A(j)$ we denote the right degree of $a$ as $r(a)$, and the left degree of $a$ by $l(a)$. Here, $l(a)=i$, $r(a)=j$. 

We will now give a generalization of this lemma.
\begin{lemma}\label{left2}  Let $f \in \sum_{1\leq i,j\leq t}A(i)\otimes  A(j)$  for some $t  \geq 1$. For each $n >t$ there exists a linear space  
 $F(f,n)\subseteq A(n)\otimes \sum_{i\geq {n\over 2t}}^{2tn}A(i)$   with
 $ \dim _{K} F(f,n)\leq 2^{2t+1}tn$ such that, for every integer $ m \geq n$,  $f^{m}$ belongs to the right ideal of $A\otimes A$ generated by elements of the form $hgf^{m-n}$, where $h\in F(f,n)$ and $g\in \sum_{0\leq i\leq t}A(i)\otimes K$.
\end{lemma}
\begin{proof} With the notation from the previous lemma, we have $f=\sum _{i=1}^{t}f_{i}$, where $f_{i}\in A(i)\otimes (A(1)+\ldots +A(t))$. We can now write 
$f_{i}=\sum_{j=1}^{t}f_{i,j}$, where $f_{i,j}\in A(i)\otimes A(j)$.  
 Observe that for every $i, j$, $l(f_{i,j})\leq t\cdot r(f_{i,j})$ and $r(f_{i,j})\leq t\cdot l(f_{i,j})$, therefore ${1\over t}\cdot l(f_{i,j})\leq r(f_{i,j})\leq t\cdot l(f_{i,j})$. Similarly if $0\leq i,j\leq t$, then for an element $g\in f_{i,j}(A(i)\otimes A(j))$ we have 
 ${1\over 2t}\cdot l(g)\leq r(g)\leq 2t\cdot l(g)$.
 It follows that if $m\in A\otimes A$ is a product of some elements $f_{i,j}$  or such elements $g$ then 
${1\over 2t}\cdot l(m)\leq r(m)\leq 2t\cdot l(m)$.  
 By the definition of the set $F(f,n)$ from the previous Lemma, and because $n>t$,  we get that any element from $F(f,n)$ there is a sum of elements $g$ with 
${l(g)\over 2t}\leq r(g)\leq 2t\cdot l(g)$.
  By the previous lemma $F(f,n)\subseteq A(n)\otimes A$ and  hence $F(f,n)\subseteq A(n)\otimes \sum_{i\geq {n\over 2t}}^{2tn}A(i),$ as required. 
Also by construction of element $g$ in thesis of  Lemma \ref{left1}  we have  $g\in \sum_{0\leq i\leq t}A(i)\otimes K$, as required. 
\end{proof}
\begin{lemma}\label{introducingT} Let $f\in H$ where $H=\sum_{1\leq i,j\leq t}A(i)\otimes A(j)$ for some $t\geq 0$, and let $n>2t$. Let $F(f,r)$ be as in Lemma \ref{left2}.
Let $T(f, n)$ be the smallest homogeneous linear space containing  $\bigcup _{0\leq i,j\leq t}(A(i)\otimes A(j))F(f,n- i)$. Then 
 $T(f,n)\subseteq A(n )\otimes \sum_{i\geq {{n-t}\over 2t}}^{2tn+t}A(i)$ and 
 $ \dim _{K} T(f,n )< 2^{4t+3}n^{3}t^{2}$. Moreover, for every $m>n$, $f^{m}, H'f^{m}\subseteq T(f,n)H'f^{m-n}(A\otimes A),$ where $H'=\sum_{0\leq i,j\leq t}A(i)\otimes A(j)$.
\end{lemma}
\begin{proof} 
By Lemma \ref{left2}, $F(f,n-i)\subseteq  A(n-i)\otimes \sum_{i\geq {n-i\over 2t}}^{2t(n-i)}A(i)$. Hence,
$F(f,n-i)(A(i)\otimes A(j))\subseteq A(n)\otimes  \sum_{i\geq {{n-t}\over 2t}}^{2tn+t}A(i)$. 
 By the definition of $T(f,n)$ we get $T(f,n)\subseteq A(n )\otimes \sum_{i\geq {{n-t}\over 2t}}^{2tn+t}A(i)$.

By Lemma \ref{left2},  $\dim _{K} F(f, n-i)<2^{2t+1} tn$,  also $\dim _{K}A(i)=2^{i}$, hence  
$\dim _{K}T(f, n)< 2^{2t+1}tn\cdot 2^{2t+2}\cdot (2tn +t)\leq 2^{4t+3}t^{2}n ^{3} $ (since $n>2t\geq 2$). 

Let $i,i'\leq t.$ Observe now that by Lemma \ref{left2} applied for $n-i$, $(A(i)\otimes A(i'))f^{m}\subseteq (A(i)\otimes A(i'))F(n-i)H'f^{m-n+i}(A\otimes A)\subseteq T(f,n)H'f^{m-n+i}(A\otimes A),$
 as required.
\end{proof} 

 The following lemmas can be proven in the same way as Lemma \ref{left2},  when we interchange the left side with the right side.
 \begin{lemma}\label{r3}  Let $f \in \sum_{1\leq i,j\leq t}A(i)\otimes  A(j)$  for some $t  \geq 1$. For each $n >t$ there exists a linear space  
 $F'(f,n)\subseteq \sum_{i\geq {n\over 2t}}^{2tn}A(i)\otimes A(n)$  with
 $ \dim _{K} F'(f, n)\leq 2^{2t+1}tn$ such that, for every integer $ m \geq n$,  $f^{m}$ belongs to the right ideal of $A\otimes A$ generated by elements of the form $hgf^{m-n}$, where $h\in F'(f,n)$ and $g\in K\otimes \sum_{0\leq i\leq t}A(i)$.
\end{lemma}

\begin{lemma}\label{introducingT'} Let $f\in H$ where $H=\sum_{1\leq i,j\leq t}A(i)\otimes A(j)$ for some $t\geq 0$, and let $n>2t$. Let $F'(f,r)$ be as in Lemma \ref{r3}.
Let $T'(f, n)$ be the smallest homogeneous linear space containing $\bigcup _{0\leq i,j\leq t}(A(i)\otimes A(j))F'(f,n- j)$. Then 
 $T'(f,n)\subseteq \sum_{i\geq {{n-t}\over 2t}}^{2tn+t }A(i)\otimes A(n)$ and 
 $ \dim _{K} T'(f,n )< 2^{4t+3}n^{3}t^{2} $. Moreover,  for all $n$, we have 
for every integer $ m \geq n$,  $H'f^{m}\subseteq T'(f,n)H'f^{m-n}$, where $H'=\sum_{0\leq i,i'\leq t}A(i)\otimes A(i')$.
\end{lemma}
\begin{lemma}\label{smaller} Let $f\in H$ where $H=\sum_{1\leq i,j\leq t}A(i)\otimes A(j)$ for some $t\geq 0$, and let $n\leq 2t$.
 Define $T'(f,n)=\sum_{i\geq n-t}^{2tn}A(i)\otimes A(n)$.
Then $ \dim _{K} T'(f,n )< 2^{2tn+1}\leq 2^{2t^{2}+1}.$ Moreover, for all $n\leq 2t$, and  
for every integer $ m \geq n$,  $H'f^{m}\in T'(f,n)H'f^{m-n}(A\otimes A)$, where $H'=\sum_{0\leq i,j\leq t}A(i)\otimes A(j)$.
\end{lemma}
\begin{proof}
 We proceed by induction on $n$, let $n=1$. If  $1\leq i',$  then $(A(i)\otimes A(i'))f^{m}\subseteq (A(i)\otimes A(1))(K\otimes A(i'-1))f^{m}\subseteq T(f,1)H'f^{m}.$
 If $i'=0$ then $(A(i)\otimes K)f^{m}=((A(i)\otimes K)f)\cdot f^{m-1}\subseteq (\sum_{i=0}^{2t}A(i)\otimes A(1))(K\otimes \sum_{i<t}A(i))f^{m-1}\subseteq T(f,1)H'f^{m}.$
 We now assume the result is true for some number $n$ (and all numbers $m\leq n$), and prove it for $n+1.$ 

 Let $i'>0$ then  $(A(i)\otimes A(i'))f^{m}\subseteq (A(i)\otimes A(1))(K\otimes A(i'-1))f^{m}.$ 
By the inductive assumption, $(A(i)\otimes A(i'-1))f^{m}\subseteq 
T'(f,n)H'f^{m-n}(A\otimes A)$. Hence $(A(i)\otimes A(i'))f^{m}\subseteq (A(i)\otimes A(1))T'(f,n)H'f^{m-n}(A\otimes A)\subseteq T'(f,n+1)H'f^{m-n}(A\otimes A).$

If $i'= 0,$ then similarly as before $(A(i)\otimes A(i'))f^{m}\subseteq ((A(i)\otimes K)f)\cdot f^{m-1}\subseteq (\sum_{j=1}^{2t}A(j)\otimes A(1))H'f^{m-1}.$ 
 By the inductive assumption $H'f^{m-1}\subseteq T'(f,n)$. Notice that $(\sum_{j=1}^{2t}A(j)\otimes A(1))T(f,n)\subseteq T'(f,n+1),$ so $(A(i)\otimes A(i'))f^{m}\subseteq  T'(f,n+1)H'f^{m-1},$ as required
\end{proof}
\section{ Appendix: Proof of Theorem \ref{S}.\\ Linear spaces $G_{1}$ and $G_{2}$}\label{appendix2}
 We now introduce a simple lemma which can be read independently from the rest of this paper. 
\begin{lemma}\label{independent}
 Let $\alpha , t ,c$ be a natural numbers with $c\leq \alpha $, and let $G_{1}, G_{2}$ be homogeneous linear subspaces with 
\[G_{1}\subseteq  \sum_{i=\alpha }^{\alpha +t}A(i)\otimes \sum _{i=c}^{ \alpha +t}A(i),
 G_{2}\subseteq  \sum _{i=c}^{\alpha +t}A(i)\otimes \sum_{i=\alpha }^{\alpha +t}A(i).\]
  Then there exist linear spaces \[T_{\alpha }(G_{1})\subseteq A(\alpha )\otimes \sum_{i=c}^{\alpha }A(i), T_{\alpha }(G_{2})\subseteq  \sum_{i=c}^{\alpha }A(i)\otimes A(\alpha ),\] and such that
  \[G_{1}\subseteq T_{\alpha }(G_{1}) (\sum _{0\leq i,j\leq t}A(i)\otimes A(j)), G_{2}\subseteq T_{\alpha }(G_{2}) (\sum _{0\leq i,j\leq t}A(i)\otimes A(j)).\] Moreover, $\dim _{K} T_{\alpha }(G_{1})\leq 2^{2t+2} \dim _{K} G_{1}$ and $\dim _{K} T_{\alpha }(G_{2})\leq 2^{2t+2} \dim _{K} G_{2}$.
 \end{lemma} 
\begin{proof} We will do a proof for the set $G_{1}$; proof for the  set $G_{2}$ is similar.  Let $c_{1}, c_{2}, \ldots , c_{\gamma }$ be a basis of  $\sum _{0\leq i,j\leq t}A(i)\otimes A(j)$.
 Notice that $\gamma \leq 2^{2t+2}$, because $\dim _{k}A(i)=2^{i}$. 
 Then we can write every  element $g\in G_{1}$ in the form  $g=\sum_{i=c}^{\gamma }g_{i}c_{i}$, where $g_{i}\in  A(\alpha )\otimes \sum _{c\leq i\leq \alpha }A(i).$
  Let $g(1), g(2), \ldots $ be a basis of the linear space $G$. We can do such a decomposition  $g(k)=\sum_{i=1}^{\gamma }g(k)_{i}c_{i}$ for every element from the basis of $G$.
 We can now put $T_{\alpha }(G_{1})$ to be a linear space spanned by all these elements $g(k)_{i}$; it is clear that set $T_{\alpha }(G_{1})$ satisfies the thesis of the theorem.
\end{proof}\\
We will now prove another  modification of Lemma $1$ from \cite{s}, and Lemma $7$ from \cite{n}.
\begin{lemma}\label{G}  Let $f \in \sum _{1\leq i,j\leq t}A(i)\otimes A(j)$  for some $t  \geq 1$. For each $n >t$ there exist homogeneous  linear spaces 
 $G_{1}(f,n)\subseteq A(n)\otimes \sum_{i\geq {n\over 2t}}^{n}A(i)$, $G_{2}(f,n)\subseteq \sum_{i\geq {n\over 2t}}^{n}A(i)\otimes A(n)$  with 
 $ \dim _{K} G_{1}(f,n), \dim _{K} G_{2}(f,n) \leq 2^{2t+6}n^{2}t^{3}$ such that, for every integer $m \geq n+1$,  $f^{m}$ belongs to the right ideal of $A\otimes A$ generated by polynomials of the form $hgf^{m-n-1}$, where $h\in G_{1}(f,n)\bigcup G_{2}(f,n)$ and $g\in \sum _{0\leq i,j\leq t} A(i)\otimes A(j)$.
\end{lemma}

\begin{proof}  Write $f=\sum _{1\leq i,j\leq t}f_{i,j}$ where $f_{i}\in A(i)\otimes A(j)$. 
  Then $f^{m}$ is a sum of terms of the form
 \[f_{i_{1}, j_{1}}f_{i_{2}, j_{2}}\ldots    f_{i_{k}, j_{k}}f_{i_{k+1}, j_{k+1}}\ldots  f_{i_{m}, j_{m}}\] 
where k denotes the largest subscript such that both $i_{1} +i_{2} +\ldots +i_{k} \leq  n$ and $j_{1}+j_{2}+\ldots +j_{k}\leq n$. We then have the decomposition:
\[f^{m} =\sum_{k\leq n}[(q_{k}+q'_{k}+q''_{k})f^{m-k}+(\sum_{0\leq i,j <t }q_{k, i,j}+ \sum_{0\leq i<t}l_{k, i}+\sum_{0\leq i<t}\bar {l}_{k, i})f^{m-k-1}]\]
  where  
\[ q_{k}=\sum_{i_{1}+\ldots +i_{k}=n; j_{1}+\ldots +j_{k}=  n }f_{i_{1}, j_{1}}\ldots f_{i_{k}, j_{k}},\]
\[ q'_{k}=\sum_{i_{1}+\ldots +i_{k}=n; j_{1}+\ldots +j_{k}< n }f_{i_{1}, j_{1}}\ldots f_{i_{k}, j_{k}},\]
\[ q''_{k}=\sum_{i_{1}+\ldots +i_{k}<n; j_{1}+\ldots +j_{k}= n }f_{i_{1}, j_{1}}\ldots f_{i_{k}, j_{k}},\]

 \[q_{k,i, j}=\sum_{i_{1}+\ldots +i_{k}=n-i, i_{k+1}>i; j_{1}+\ldots +j_{k}=n-j, j_{k+1}>j }f_{i_{1}, j_{1}}\ldots f_{i_{k}, j_{k}}f_{i_{k+1}, j_{k+1}}\]
\[l_{k,i,}=\sum_{i_{1}+\ldots +i_{k}=n-i, i_{k+1}>i; j_{1}+\ldots +j_{k}+ j_{k+1}\leq  n }f_{i_{1}, j_{1}}\ldots f_{i_{k}, j_{k}}f_{i_{k+1}, j_{k+1}}\]
\[ \bar {l}_{k, j}=\sum_{j_{1}+\ldots +j_{k}=n-i, j_{k+1}>i; i_{1}+\ldots +i_{k}+i_{k+1}\leq n }f_{i_{1}, j_{1}}\ldots f_{i_{k}, j_{k}}f_{i_{k+1}, j_{k+1}}\]
 
 Notice that, similarly to Lemma \ref{left2}, we get that 
\[q,q',q'',q_{k,i,j}, l_{k,i}, {\bar l}_{k,i}\subseteq  (\sum_{i=0}^{t}A(n+i)\otimes \sum_{i\geq {n\over 2t}}^{2tn}A(i))+ (\sum_{i\geq {n\over 2t}}^{2tn}A(i)\otimes \sum_{i=0}^{t}A(n+i)).\] On the other hand, by construction $q,q',q'',q_{k,i,j}, l_{k,i}, {\bar l}_{k,i}\subseteq  \sum_{i, j\leq n+t}A(i)\otimes A(j)$. It follows that \[q,q',q'',q_{k,i,j}, l_{k,i}, {\bar l}_{k,i}\subseteq  (\sum_{i=0}^{t}A(n+i)\otimes \sum_{i\geq {n\over 2t}}^{n+t}A(i))+(\sum_{i\geq {n\over 2t}}^{n+t}A(i)\otimes \sum_{i=0}^{t}A(n+i)).\]

 Let $G$ be a homogeneous linear space spanned  by all  homogeneous components of $q_{k}$, $q'_{k}$, $q''_{k}$, $q_{k,i,j}$, $l_{k,i}$, ${\bar l}_{k,i}$, then $dim _{K} G\leq 6nt^{2}\cdot (n+t)t=6n(n+t)t^{3}$. 
 Let $c$ be the smallest natural number which equals  at least ${n\over 2t}$ and denote $\alpha =n$. Then there are linear spaces   \[G_{1}\subseteq ( \sum_{i=\alpha }^{\alpha +t}A_{i}\otimes \sum _{i=c}^{ \alpha +t}A_{i}), G_{2}\subseteq  (\sum _{i=c}^{\alpha +t}A
_{i}\otimes \sum_{i=\alpha }^{\alpha +t}A_{i})\] such that $G=G_{1}+G_{2}$ and $\dim _{K} G_{1}, \dim _{K} G_{2}\leq \dim _{K} G$.  We can now take  
 $G_{1}(f,n) = T_{n}(G_{1})$,  $G_{2}(f,n) = T_{n}(G_{2})$ as in the Lemma \ref{independent} applied for $\alpha =n$. 
 By Lemma \ref{independent},  $\dim _{K} G_{1}(f,n), \dim _{K} G_{2}(f,n)\leq 2^{2t+2}\dim _{K} G\leq 2^{2t+2}6n(n+t)t^{3}< 2^{2t+6}n^{2}t^{3}$, as required.
\end{proof}

\section{Appendix: Proof of Theorem \ref{S}. \\ Introducing sets $Q(S,n,m)$}\label{appendix3}
  
 Before introducing sets $Q(S,n,m)$, we will prove the following lemma.

\begin{lemma}\label{E}  Let $f \in H$ where $H= \sum_{1\leq i,j\leq t}A(i)\otimes  A(j)$  for some $t  \geq 1$. Let $n_{1}>t$, $n_{2}-3n_{1}>t$, $n_{3}>2tn_{2},$ $m>8n_{3}+6tn_{1}$ be natural numbers.
 Let $G_{1}(f,n)$, $G_{2}(f,n)$ be as in Lemma \ref{G}, and $T(f,n)$, $T'(f,n)$ be as in Lemmas \ref{introducingT} , \ref{introducingT'} and \ref{smaller}. 
 Then the two-sided  ideal of $A\otimes A$ generated by $f^{m}$ belongs to the right ideal of $A\otimes A$ generated by elements from   
\[\sum_{1\leq k, a_{2}< j}((A(k\cdot n_{3}-n_{2})\otimes A(j))T(f,n_{2}-n_{1})T(f,n_{1})(E(f,n_{1},n_{2})\cap E'(f,n_{1}, n_{2}))\] where 
 $E(f,n_{1},n_{2})= T(f,n_{1})T(f,n_{2}-n_{1})(A\otimes A),$ 
 $E'(f,n_{1}, n_{2})=S(f,n_{1},n_{2})+S'(f,n_{1}, n_{2}),$  
\[S(f,n_{1},n_{2})= G_{1}(f, n_{2}-n_{1})T(f,n_{1})(A\otimes A) \]
\[S'(f,n_{1}, n_{2})=G_{2}(f, n_{2}-n_{1})\bigcap _{i=0}^{2n_{1}}(T'(f,i)T'(f, n_{1})(A\otimes A)).\]
 Moreover, linear spaces $E(f, n_{1},n_{2})$, $E'(f, n_{1},n_{2})$, $S(f, n_{1},n_{2})$, $S'(f, n_{1},n_{2})$ are homogeneous.
\end{lemma}
 \begin{proof} The ideal generated by $f^{m}$ is spanned by elements $(A(k\cdot n_{3}-n_{2}-v)\otimes A(j))f^{m}$, for some $n_{3}\leq v<2n_{3},$ and some $0\leq k,j$.  
 By Lemma \ref{left2}, $f^{m}$ belongs to the right ideal of $A\otimes A$ generated by elements of the form $hgf^{m-v}$, where $h\in F(f,v)$ and $g\in \sum_{0\leq i\leq t}A(i)\otimes K$. Recall that ${v\over t}\geq {n_{3}\over t}>a_{2}.$
 Therefore, the ideal generated by $f^{m}$ belongs to the right ideal generated by elements from $\sum_{1\leq k,n_{2}\leq i}((A(k\cdot n_{3}-n_{2})\otimes A(i))Hf^{m-v}(A\otimes A)$ where $H=\sum_{0\leq i,j\leq t}A(i)\otimes A(j)$.
  By Lemma \ref{introducingT}, we get 
$Hf^{m-v}\subseteq T(f,n_{2}-n_{1})Hf^{m-v-n_{2}+n_{1}}(A\otimes A)$. Applying Lemma \ref{introducingT} again, we get 
$Hf^{m-v-n_{2}+n_{1}}\subseteq T(f,n_{1})Hf^{m-v-n_{2}}(A\otimes A)$.
 Therefore, the two sided ideal of $A\otimes A$ generated by $f^{m}$ is contained in 
$\sum_{1\leq , a_{2}<i}(A(k\cdot n_{3}-n_{2})\otimes A(i))T(f,n_{2}-n_{1})T(f,n_{1})Hf^{m-n_{2}-v}(A\otimes A)$.
  Recall that $v\leq 2n_{3}$, so it suffices to prove that $Hf^{m-n_{2}-2n_{3}}\subseteq E(f,n_{1}, n_{2})\cap E'(f,n_{1}, n_{2})$. 
 By Lemma \ref{introducingT}, applied for $n=n_{1}$ and then for $n=n_{2}-n_{1}$  we get 
$Hf^{m-n_{2}-2n_{3}}\subseteq T(f, n_{1})T(f,n_{2}-n_{1})(A\otimes A),$
so $Hf^{m-n_{2}-2n_{3}}\subseteq E(f,n_{1},n_{2})$. 

 On the other hand, by Lemma \ref{G}, applied to $n=n_{2}-n_{1}$ we get that   $Hf^{m-n_{2}-2n_{3}}\subseteq (G_{1}(f,n_{2}-n_{1})+G_{2}(f,n_{2}-n_{1}))Hf^{m-2n_{2}-2n_{3}+n_{1}-1}(A\otimes A).$ 
 
 Clearly, $m-2n_{2}-2n_{3}+n_{1}-1\geq m-4n_{3}$. 
By Lemma \ref{introducingT}, applied to $n=n_{2}$, we get 
 $G_{1}(f,n_{2}-n_{1})Hf^{m-4n_{3}} \subseteq G_{1}(f,n_{2}-n_{1})T(f,n_{1})(A\otimes A)=S(f,n_{1},n_{2})(A\otimes A)\subseteq E(f,n_{1}, n_{2})$. 
 
 If $i>2t$ then we apply Lemma \ref{introducingT'} for $n=i$; if $i\leq 2t$ then we apply Lemma \ref{smaller}
 and we get 
$Hf^{m-4n_{3}}\subseteq T'(f,i)H'f^{m-4n_{3}-i}(A\otimes A)$ where $H'=\sum_{0\leq l,l'\leq t}A(l)\otimes A(l').$
 We then apply \ref{introducingT'}
 for $n=n_{1}$ to get 
$Hf^{m-4n_{3}}\subseteq T'(f,i)T'(f, n_{1})(A\otimes A)$ for every $n_{1}\leq i\leq 2n_{1}$.
 Therefore $G_{2}(f,n_{2}-n_{1})Hf^{m-4n_{3}}\subseteq  G_{2}(f, n_{2}-n_{1})\bigcap_{i=n_{1}}^{2n_{1}}( T'(f,i)T'(f, n_{1})(A\otimes A))\subseteq E'(f,n_{1}, n_{2}).$

 Moreover, sets $E(f,n_{1},n_{2})$, $E('f,n_{1},n_{2})$, $S(f,n_{1},n_{2}$, $S'(f,n_{1},n_{2}$ are homogeneous, because by Lemmas \ref{introducingT}, \ref{introducingT'} sets $T(f,n_{1})$ and $T(f, n_{2})$ are homogeneous.
\end{proof}

Let $1\leq \alpha $, $0< n_{1}<n_{2}$ be natural numbers and let $F\subseteq A(n_{1})\otimes A(\alpha )$. Recall that 
 \[Q(F, n_{1}, n_{2})=\sum_{k=0}^{\infty }\sum_{j\in W(n_{1}, n_{2})}((A(k\cdot n_{2}-n_{1})+A(k\cdot n_{2}))\otimes A(j))F(A\otimes A),\] where 
$j\in W(n_{1}, n_{2})$ if and only if  the interval $[j,j+\alpha ]$ is disjoint with all intervals $[kn_{2}-n_{1}, kn_{2}+n_{1}]$ for $k=0, 1,2, \ldots $.

 The following lemma will be used in the proof of Lemma \ref{final}.
\begin{lemma} \label{1}
 Given a natural number $n$, and a linear space $D\subseteq A\otimes \sum_{i=1}^{\infty } A(n+i)$ let $Z_{n}(D)$ to be the smallest linear space that $D\subseteq Z_{n}(D)(K\otimes A(n))$. 
  Similarly, let ${\bar Z}_{n}(D)$ to be the smallest linear space that $D\subseteq (K\otimes A(n)){\bar Z}_{n}(D)$.
Then   $\dim _{K}Z_{n}(D), \dim _{K}{\bar Z}_{n}(D)\leq 2^{n}\dim_{K}D$.
\end{lemma}
\begin{proof} Similar to Lemma \ref{independent}.
 \end{proof}

\begin{lemma}\label{final} Let $f \in H$ where $H= \sum_{1\leq i,j\leq t}A(i)\otimes  A(j)$  for some $t  \geq 1$. Let $0<n_{1}<n_{2}< n_{3}$ be natural numbers such that $n_{1}$ divides $n_{2}$, $n_{2}$ divides $n_{3}$ and $10tn_{1}< n_{2}$, $10tn_{2}< n_{3}$, $10tn_{3}<m$, $10t<n_{1}$. Let $U(n_{1}, n_{3})=\bigcup_{k=1}^{\infty }[kn_{3}-n_{1}, kn_{3}+n_{1}]$. Let $G_{1}(f,n)$, $G_{2}(f,n)$ be as in Lemma \ref{G}, and $T(f,n)$, $T'(f,n)$ be as in Lemmas \ref{introducingT}, \ref{introducingT'} and \ref{smaller}, and  $E'(f,n_{1}, n_{2})$ be as in Lemma \ref{E}.
 Then there is set $F(f,n_{1}, n_{2}, n_{3})\subseteq A(n_{2})\otimes \sum_{i=1}^{2tn_{2}+2t}A(i)$ with $\dim _{K} F(f,n_{1}, n_{2}, n_{3})\leq 2^{40tn_{1}}n_{2} $ and such that 
 the two-sided  ideal of $A\otimes A$ generated by $f^{m}$ belongs to the right ideal of $A\otimes A$ generated by elements from  $Q+P(A(n_{3})\otimes A(n_{3}))$ where $Q=Q(T(f,n_{1}), n_{1}, n_{2})+Q(F(f,n_{1}, n_{2}, n_{3}), n_{2}, n_{3}),$ and  where
\[P=\sum_{k=1}^{\infty }\sum_{j\in U(n_{1}, n_{3})}(A(k\cdot n_{3})\otimes A(j))E'(f,n_{1}, n_{2})(A\otimes A).\]
 Moreover, if $n_{2}>100tn_{1}$ then $F(f,n_{1}, n_{2}, n_{3})\subseteq A(n_{2})\otimes \sum_{i>{n_{2}\over 10t}}^{2tn_{2}+2t}A(i).$
\end{lemma}
\begin{proof}   Let  $D= T(f,n_{2}-n_{1})T(f,n_{1})+T(f,n_{1})T(f,n_{2}-n_{1})+T(f,n_{2}) $, 
\[F(f,n_{1}, n_{2}, n_{3})=D+Z_{5tn_{1}}(D)+{\bar Z}_{5tn_{1}}(D),\] where $Z_{n}(D)$, ${\bar Z}_{n}(D)$  are as in Lemma \ref{1}. 
By Lemma \ref{1}, we get $\dim _{K} F(f,n_{1}, n_{2}, n_{3})\leq (1+2^{5tn_{1}+1}) \dim _{K} D \leq  2^{40tn_{1}}n_{2}$.
 By Lemma \ref{introducingT}, $T(f,n)\subseteq A(n)\otimes \sum_{i>{n-t\over 2t}}^{2nt+t}A(i)$, so 
   $F(f,n_{1}, n_{2}, n_{3})\subseteq A(n_{2})\otimes \sum_{i>{{n_{2}-2t}\over 2t}-5tn_{1}}^{2tn_{2}+2t}A(i).$
Notice that  if $n_{2}>100tn_{1}$ then $F(f,n_{1}, n_{2}, n_{3})\subseteq A(n_{2})\otimes \sum_{i>{n_{2}\over 10t}}^{2tn_{2}+2t}A(i).$

  Let $f_{1}\in T(f,n_{2}-n_{1})$, $f_{2}\in T(f,n_{1})$, and let  $j<n_{3}$.
 
  By Lemma $\ref{E}$ the ideal generated by $f^{m}$ is contained in the right ideal of $A\otimes A,$ $\sum_{k=1}^{\infty }\sum_{ j=0}^{\infty }((A(k\cdot n_{3}-n_{2})\otimes A(j))T(f,n_{2}-n_{1})T(f,n_{1})(E(f,n_{1},n_{2})\cap E'(f,n_{1}, n_{2}, n_{3}))$.
 Let $f_{1}\in T(f, n_{2}-n_{1})$, $f_{2}\in T(f,n_{1})$, $g_{3}\in E(f,n_{1},n_{2})\cap E'(f,n_{1}, n_{2}, n_{3})$.
 By Lemma \ref{introducingT}, there are ${n_{1}-t\over 2t}\leq \alpha _{1}\leq 2tn_{1}+t$, ${n_{2}-n_{1}-t\over 2t}\leq \alpha _{2}\leq 2t(n_{2}-n_{1})+t$, such that 
 $f_{1}\in A(n_{2}-n_{1})\otimes A(\alpha _{1})$, $f_{2}\in A(n_{1})\otimes A(\alpha _{2})$.
  
Assume first that  $j+\alpha _{1}+\alpha _{2}\in [kn_{3}-n_{1}, kn_{3}+n_{1}]$, for some $k$. Notice that the case $k=0$ is not possible as $j>a_{2}$ by Lemma \ref{E}. It follows that 
    $(A(i\cdot n_{3}-n_{2})\otimes A(j))f_{1}f_{2}g_{3}\subseteq \sum_{j'\in U(n_{1}, n_{3}) }(A(n_{3})\otimes A(j'))g_{3}(A\otimes A)\subseteq P,$ because $g_{3}\in E'(f,n_{1}, n_{2})$.

 If $j+\alpha _{1}+\alpha _{2}\notin [kn_{3}-n_{1}, kn_{3}+n_{1}]$, for any $k$,  it suffices to prove that 
   $(A(n_{3}-n_{2})\otimes A(j)f_{1}f_{2}T(f,n_{2}-n_{1})T(f,a_{1})(A(n_{3})\otimes A(n_{3}))\subseteq Q$ ( notice that $A(n_{3})$ at the end  comes from the fact that $f^{m}$ has large degree, and all considered sets are homogeneous).

   Let $f_{3}\in T(f,n_{2}-n_{1})$, $f_{4}\in T(f, n_{1})$, then by Lemma \ref{introducingT}, $f_{3}\in A(n_{1})\otimes A(\alpha _{3})$, $f_{4}\in A(n_{2}-n_{1})\otimes A(\alpha _{4})$, 
 for some  ${n_{1}-t\over 2t}\leq \alpha _{3}\leq 2tn_{1}+t,$ ${{n_{2}-n_{1}-t}\over 2t}\leq \alpha _{4}\leq 2t(n_{2}-n_{1})+t.$ 
 It suffices now to show that $(A(n_{3}-n_{2})\otimes A(j))f_{1}f_{2}f_{3}f_{4}(A(3tn_{3})\otimes A(3tn_{3}))\subseteq Q$.

It suffices now to consider the  the following cases:

{\bf Case 1.} $j+\alpha_{1}+\alpha _{2}\in [kn_{3}+n_{2}+1, (k+1)n_{3}-n_{2}-1]$ for some $k\geq 0$. Then we have the following two subcases.

{\bf Case 1.a} $j+\alpha_{1}+\alpha _{2}\in [kn_{3}+n_{2}+1, (k+1)n_{3}-4tn_{2}]$ for some $k\geq 0$. Denote $i=j+\alpha _{1}+\alpha _{2},$ then $i+\alpha _{3}, \alpha _{4}\leq (k+1)n_{3}-n_{2}-1,$ since $\alpha _{1}+\alpha _{2}\leq 2tn_{2}=2t.$ 
 Therefore,  $(A(n_{3}-n_{2})\otimes A(j))f_{1}f_{2}f_{3}f_{4}\subseteq \sum_{i:i, i+\alpha _{3}+\alpha _{4}\in [kn_{3}+n_{2}+1, (k+1)n_{3}-n_{2}-1]}(A(n_{3})\otimes A(i))f_{3}f_{4}(A\otimes A)\subseteq Q(D,n_{2},n_{3})\subseteq Q$, since $D\subseteq F(f, n_{1}, n_{2}, n_{3})$.

{\bf Case 1.b} $j+\alpha_{1}+\alpha _{2}\in [kn_{3}+4tn_{2}, (k+1)n_{3}-n_{2}-1]$ for some $k\geq 0,$ then $j\geq kn_{3}+n_{2}+1,$ since $\alpha _{1}+\alpha _{2}\leq 2n_{2}+2t.$ 
 Hence, $((A(n_{3}-n_{2})\otimes A(j))f_{1}f_{2}f_{3}f_{4}\subseteq \sum_{j:j,j+\alpha _{1}+\alpha _{2}\in [kn_{3}+n_{2}+1, (k+1)n_{3}-n_{2}-1] }(A(n_{3})\otimes A(j))f_{1}f_{2}(A\otimes A)\subseteq Q(D,n_{2},n_{3})\subseteq Q$.

{\bf Case 2.}  $j+\alpha _{1}+\alpha _{2}\in [kn_{3}+n_{1}+1, kn_{3}+n_{2}]$, for some $k$. We have the following two subcases. 

{\bf Case 2a.} $j+\alpha _{1}+\alpha _{2}\in [kn_{3}+n_{1}+1, kn_{3}+n_{2}-4tn_{1}]$, for some $k$. In this case 
 $(A(n_{3}-n_{2})\otimes A(j))f_{1}f_{2}f_{3}f_{4}\subseteq \sum_{i\in [kn_{3}+n_{1}+1, kn_{3}+n_{2}-4tn_{1}] }(A(n_{3})\otimes A(i))f_{3}(A\otimes A)\subseteq Q(T(f,n_{1}),n_{1},n_{2})\subseteq Q$, because $\alpha _{3}\leq 2tn_{1}+t\leq 4tn_{1}-1,$
 and $f_{3}\in T(f, n_{1}).$

 {\bf Case 2b.} $j+\alpha _{1}+\alpha _{2}\in [kn_{3}+n_{2}-4tn_{1}, kn_{3}+n_{2}]$, for some $k$. 
 Denote $i=j+\alpha _{1}+\alpha _{2},$  then $i+\alpha _{3}+\alpha _{4}\leq (k+1)n_{3}-n_{2}-1.$

 By Lemma \ref{independent2}, we have  
$(A(n_{3}-n_{2})\otimes A(j))f_{1}f_{2}f_{3}f_{4}\subseteq C$ where $C=
\sum_{i:i,i+\alpha _{3}+\alpha _{4}\in [kn_{3}+n_{2}-4tn_{1}, (k+1)n_{3}-n_{2}-1],0\leq k'}(A(k'\cdot n_{3})\otimes A(i))f_{3}f_{4}(A\otimes A).$

Denote  $q=\alpha _{3}+\alpha _{4}-5tn_{1},$ then ${\bar  Z}_{5tn_{1}}(K\xi(f_{3}f_{4}))\subseteq A(n_{2})\otimes A(q).$
 By Lemma \ref{1}, applied for $n=5tn_{1}$ and for the linear space $Kf_{3}f_{4}$, we get  
 \[C\subseteq \sum_{i: i, i+q\in [kn_{3}+n_{2}+tn_{1}-1,( k+1)n_{3}-n_{2}-1],0\leq k'}(A(k'\cdot n_{3})\otimes A(i))){\bar Z}_{5tn_{1}}(Kf_{3}f_{4})(A\otimes A),\] so  
$C\subseteq Q(F(f,n_{1}, n_{2}, n_{3}),n_{2},n_{3})\subseteq Q$, because $f_{3}f_{4}\in D$,  so ${\bar Z}_{5tn_{1}}(Kf_{3}f_{4})\subseteq F(f, n_{1}, n_{2}, n_{3}).$

{\bf Case 3.}  $j+\alpha _{1}+\alpha _{2}\in [kn_{3}-n_{2}, kn_{3}-n_{1}-1]$, for some $k$. We have the following two subcases.

{\bf Case 3a.} $j+\alpha _{1}+\alpha _{2}\in [kn_{3}-n_{2}+4tn_{1}, kn_{3}-n_{1}-1]$, for some $k$. 

Denote, $i=j+\alpha _{1}$, then $i+\alpha _{2}=j+\alpha _{1}+\alpha _{2}\leq kn_{3}-n_{1}-1$ and $i>(j+\alpha _{1}+\alpha _{2})-\alpha _{2}\geq (kn_{3}-n_{2}+4tn_{1})-(2tn_{1}+2t)>kn_{3}-n_{2}+n_{1}+1=1.$
 Therefore the  interval $[i, i+\alpha _{2}$ is disjoint with any interval $[k'n_{2}-n_{1}, k'n_{2}+n_{1}]$ for all $k'.$ 
Hence,  $(A(n_{3}-n_{2})\otimes A(j))f_{1}f_{2}f_{3}f_{4}\subseteq \sum _{k=0}^{\infty }\sum_{i: i+\alpha _{2}\in [kn_{3}-n_{2}+n_{1}+1,  kn_{3}-n_{1}-1]}(A(n_{3}-n_{1})\otimes A(i))f_{2}(A\otimes A),$ hence 
$(A(n_{3}-n_{2})\otimes A(j))f_{1}f_{2}f_{3}f_{4}\subseteq Q(T(f,n_{1}),n_{1},n_{2})\subseteq Q$.

 {\bf Case 3b.} $j+\alpha _{1}+\alpha _{2}\in [kn_{3}-n_{2}, kn_{3}-n_{2}+4tn_{1}]$, for some $k$.
 Then, $j\geq kn_{3}-n_{2}-2tn_{2}-2t,$ hence $j,j+\alpha _{1}+\alpha _{2}\in [(k-1)n_{3}+n_{2}+1, kn_{3}-n_{2}+4tn_{1}].$
 By Lemma \ref{independent2}, applied for the linear space $Kf_{1}f_{2}$, $p=5tn_{1}$  we have 
$(A(n_{3}-n_{2})\otimes A(j))f_{1}f_{2}f_{3}f_{4}\subseteq P_{1}$ where 
$P_{1}=\sum_{j:j, j+\alpha_{1}+\alpha _{2}\in [ (k-1)n_{3}+n_{2}+1,kn_{3}-n_{2}+4tn_{1}] }(A(n_{3}-n_{2})\otimes A(i))f_{1}f_{2}(A(3tn_{3})\otimes A(3tn_{3})).$ Denote  $q=\alpha _{1}+\alpha _{2}-5tn_{1},$ then ${Z}_{5tn_{1}}(K\xi(f_{1}f_{2}))\subseteq A(n_{2})\otimes A(q).$ Observe that $P_{1}\subseteq P_{2},$ where 
$P_{2}=\sum_{j: j,j+q\in [ (k-1)n_{3}+n_{2}+1,kn_{3}-n_{2}-1] }(A(n_{3}-n_{2})\otimes A(i))Z_{5tn_{1}}(Kf_{1}f_{2})(A(3tn_{3})\otimes A(3tn_{3}))$.
  Notice that $f_{1}f_{2}\in D$, and so $ Z_{5tn_{1}}(Kf_{1}f_{2})\subseteq F(f, n_{1}, n_{2},n_{3})$. 
By Lemma \ref{independent2}, $f_{1}f_{2}\subseteq A(n_{2})\otimes A$, hence 
$P_{2}\subseteq Q(F(f,n_{1}, n_{2}), n_{2}, n_{3})\subseteq Q$.
\end{proof}

Let $1\leq \alpha $, $0< n_{1}<n_{2}$ be natural numbers and let $F\subseteq  A(\alpha )\otimes A(n_{1})$. Recall that 
 \[Q'(F, n_{1}, n_{2})=\sum_{k=0}^{\infty }\sum_{j\in W(n_{1}, n_{2})}(A(j)\otimes (A(k\cdot n_{2})+A(k\cdot n_{2}-n_{1}))F(A\otimes A),\] where 
$j\in W(n_{1}, n_{2})$ if and only if $[j,j+\alpha  ]$ is disjoint with all the intervals $[kn_{2}-n_{1}, kn_{2}+n_{1}]$ for $k=0, 1, 2, \ldots $.

Let $x_{i}\in \{x,y\}$ be generators of $A$, denote $(x_{1}x_{2}\cdots x_{n})^{op}=x_{n}x_{n-1}\cdots x_{1}$.
Let $a_{i}, b_{i}\in A$. For $r=\sum_{i}a_{i}\otimes b_{i}$ define $\xi(r)=\sum_{i}a_{i}\otimes (b_{i})^{op}.$  
 The following Lemma is similar to Lemma \ref{final}

 Recall that if $a\in A(n)\otimes A(m)$, for some $m,n$, then $l(a)=n$, $r(a)=m$. 
  We are now ready to prove the two  main theorems in this section. 

\begin{lemma}\label{finalop} Let $f \in H$ where $H= \sum_{1\leq i,j\leq t}A(i)\otimes  A(j)$  for some $t  \geq 1$. Let $n_{1},$ $n_{2}, n_{3}$ be natural numbers such that $n_{1}$ divides $n_{2}$, $n_{2}$ divides $n_{3}$ and $10tn_{1}< n_{2}$, $10tn_{2}< n_{3}$, $10tn_{3}<m$, $10t<n_{1}$. Let $U(n_{1}, n_{3})=\bigcup_{k=1}^{\infty }[kn_{3}-n_{1}, kn_{3}+n_{1}]$. Let $G_{1}(f,n)$, $G_{2}(f,n)$ be as in Lemma \ref{G}, and $T(f,n)$, $T'(f,n)$ be as in Lemmas \ref{introducingT} and \ref{introducingT'} and  $E'(f,n_{1}, n_{2})$ be as in Lemma \ref{E} (and let $\xi (r)$ be defined as above for any $r\in A\otimes A$.) 
 Then there is set $F'(f,n_{1}, n_{2}, n_{3})\subseteq A(n_{2})\otimes \sum_{i=1}^{2tn_{2}+2t}A(i)$ with $\dim _{K} F(f,n_{1}, n_{2}, n_{3})\leq 2^{40tn_{1}}n_{2} $ and such that 
 the two-sided  ideal of $A\otimes A$ generated by $\xi(f^{m})$ belongs to the right ideal of $A\otimes A$ generated by elements from  $Q'+P'$ where $Q'=Q(\xi (T(f,n_{1})), n_{1}, n_{2}))+Q(F'(f,n_{1}, n_{2}, n_{3}), n_{2}, n_{3}),$ and 

\[P'=\sum_{ \alpha=0}^{ \infty }\sum_{k=1}^{\infty }\sum_{j+\alpha \in U(n_{1}, n_{3})}(A(k\cdot n_{3})\otimes A(j))\xi(E'(f,n_{1}, n_{2},\alpha ))(A\otimes A),\]

 where $E'(f,n_{1}, n_{2})=\sum_{i=0}^{\infty } E'(f,n_{1}, n_{2},\alpha )$, with $E'(f,n_{1}, n_{2},\alpha)\subseteq  A\otimes A(\alpha)$. 
Moreover, if $n_{2}>100tn_{1}$ then $F(f,n_{1}, n_{2}, n_{3})\subseteq A(n_{2})\otimes \sum_{i>{n_{2}\over 10t}}^{2tn_{2}+2t}A(i).$

\end{lemma}
\begin{proof}   Let  $D= \xi (T(f,n_{2}-n_{1})T(f,n_{1})+T(f,n_{1})T(f,n_{2}-n_{1})+T(f,n_{2})) $,   
\[F'(f,n_{1}, n_{2}, n_{3})=D+Z_{5tn_{1}}(D)+{\bar Z}_{5tn_{1}}(D),\] where $Z_{n}(D)$, ${\bar Z}_{n}(D)$  are as in Lemma \ref{1}.  
By Lemma \ref{1}, we get $\dim _{K} F'(f,n_{1}, n_{2}, n_{3})\leq (1+2^{5tn_{1}+1})\dim _{K} D \leq  2^{40tn_{1}}n_{2}.$
 By Lemma \ref{introducingT}, $T(f,n), \xi (T(f,n))\subseteq A(n)\otimes \sum_{i>{n-t\over 2t}}^{2nt+t}A(i)$, hence 
   $F'(f,n_{1}, n_{2}, n_{3})\subseteq A(n_{2})\otimes \sum_{i>{{n_{2}-2t}\over 2t}-5tn_{1}}^{2tn_{2}+2t}A(i).$
Notice that  if $n_{2}>100tn_{1}$ then $F'(f,n_{1}, n_{2}, n_{3})\subseteq A(n_{2})\otimes \sum_{i>{n_{2}\over 10t}}^{2tn_{2}+2t}A(i).$

  Let $f_{1}\in T(f,n_{2}-n_{1})$, $f_{2}\in T(f,n_{1})$, and let  $j<n_{3}$.
 By Lemma \ref{introducingT}, $f_{1}\in A(n_{2}-n_{1})\otimes A(\alpha _{1})$, $f_{2}\in A(n_{1})\otimes A(\alpha _{2})$, 
for some $ \alpha _{1}\leq 2tn_{1}+t,$ ${n_{2}-n_{1}-t\over 2t}\leq \alpha _{2}\leq 2t(n_{2}-n_{1})+t.$
 
  Let $J$ denote the ideal generated by $\xi(f^{m})$ in $A\otimes A$, and let $J'$ denotes the ideal generated by $f^{m}$ in $A\otimes A$. By Lemma $\ref{E}$, $\xi (J')\subseteq \sum_{k, j=0}^{\infty }\xi ((A(k\cdot n_{3}-n_{2})\otimes A(j))T(f,n_{2}-n_{1})T(f,n_{1})(E(f,n_{1},n_{2})\cap E'(f,n_{1}, n_{2}, n_{3})))(A\otimes A)$. Hence, $J\subseteq \sum_{k, j'=0}^{\infty }((A(k\cdot n_{3}-n_{2})\otimes A(j'))\xi (T(f,n_{2}-n_{1})T(f,n_{1})(E(f,n_{1},n_{2})\cap E'(f,n_{1}, n_{2}, n_{3})))(A\otimes A)$. 
Recal that sets $E(f,n_{1},n_{2})$ and $ E'(f,n_{1}, n_{2})$ are homogeneous, and hence they are generated by homogeneous elements. 

Let $f\in A(s)\otimes A(j)$ be such that $f\in E(f,n_{1},n_{2})\cap E'(f,n_{1}, n_{2}, n_{3})$. 
    Assume first that  $j+j'\in [k'n_{3}-n_{1}, k'n_{3}+n_{1}]$, for some $k'$. Notice that the case $k'=0$ is not possible as elements from  $E(f,n_{1},n_{2})$ have right degree larger than $n_{1}$. It follows that  
    $(A(k\cdot n_{3}-n_{2})\otimes A(j'))\xi(f_{1}f_{2}f)\subseteq 
   (A(k\cdot n_{3})\otimes A(j')\xi(f)(A\otimes A)\subseteq 
P',$ because $\xi(f)\in \xi(E'(f,n_{1}, n_{2},j))$ and $j+j'\in U(n_{1}, n_{3}).$

  Assume now that $j+j'\notin [k'n_{3}-n_{1}, k'n_{3}+n_{1}]$, for any $k'$. Recall that $f\in E(f, n_{1}, n_{2}),$ so  
 $f=f_{3}f_{4}f_{5}$ for some $f_{3}\in A(n_{1})\otimes A(\alpha _{3})$, $f_{4}\in A(n_{2}-n_{1})\otimes A(\alpha _{4})$, $f_{5}\in A(\beta )\otimes A(\alpha _{5})$ 
 for some $\alpha _{3}\leq 2tn_{1}+t,$ ${{n_{2}-n_{1}-t}\over 2t}\leq \alpha _{4}\leq 2t(n_{2}-n_{1})+t,$  and some $\beta, \alpha _{5}.$  
 Observe that $\alpha _{3}+\alpha _{4}+\alpha _{5}=j$ because $f_{3}f_{4}f_{5}=f$.
 
 Notice that it suffices to show that $(A(n_{3}-n_{2})\otimes A(j'))\xi(f_{1}f_{2}f_{3}f_{4}f_{5})\subseteq Q.$ We have the following cases:

{\bf Case 1.} $j+j'\in [kn_{3}+n_{2}+1, (k+1)n_{3}-n_{2}]$ for some $k\geq 0$. Then we have the following two subcases.

{\bf Case 1.a} $j+j'\in [kn_{3}+n_{2}+1, (k+1)n_{3}-4tn_{2}]$ for some $k\geq 0$.
 Then, $(A(n_{3}-n_{2})\otimes A(j'))\xi(f_{1}f_{2}f_{3}f_{4}f_{5})\subseteq \sum_{kn_{3}+n_{2}<j+j', j+j'+\alpha _{1}+\alpha _{2}< (k+1)n_{3}-n_{2}}(A(n_{3}-n_{2})\otimes A(j+j'))\xi(f_{1}f_{2})(A\otimes A)\subseteq Q(D,n_{2},n_{3})\subseteq Q$, since $\xi(f_{1}f_{2})\in D\subseteq F(f, n_{1}, n_{2}, n_{3})$.

{\bf Case 1.b} $j+j'\in [kn_{3}+4tn_{2}, (k+1)n_{3}-n_{2}]$ for some $k\geq 0$.
 We have, $((A(n_{3}-n_{2})\otimes A(j))\xi(f_{1}f_{2}f_{3}f_{4}f_{5})\subseteq \sum_{kn_{3}+2n_{2}<j+j',j+j'-\alpha _{3}-\alpha _{4})<(k+1)n_{3}-n_{2} }(A(n_{3})\otimes A(j+j'-\alpha _{3}-\alpha _{4})\xi(f_{3}f_{4})(A\otimes A)\subseteq Q(D,n_{2},n_{3})\subseteq Q,$
 since $\xi(f_{3}f_{4})\in D.$

{\bf Case 2.}  $j+j'\in [kn_{3}+n_{1}+1, kn_{3}+n_{2}]$, for some $k$. We have the following two subcases. 

{\bf Case 2a.} $j+j'\in [kn_{3}+n_{1}+1, kn_{3}+n_{2}-4tn_{1}]$, for some $k$. In this case 
 $(A(n_{3}-n_{2})\otimes A(j))\xi(f_{1}f_{2}f_{3}f_{4}f_{5})\subseteq \sum_{j,j':j+j'\in [kn_{3}+n_{1}+1, kn_{3}+n_{2}-4tn_{1}] }(A(n_{3}-n_{1})\otimes A(i))\xi(f_{2})(A\otimes A)\subseteq Q(\xi(T(f,n_{1})),n_{1},n_{2})\subseteq Q$ (because $\alpha _{3}\leq 2tn_{1}+t\leq 4tn_{1}-n_{1}-1,$ so $j+j',j+j'+\alpha _{3}\in [kn_{3}+n_{1}+1, kn_{3}+n_{2}-n_{1}-1]$).

 {\bf Case 2b.} $j+j'\in [kn_{3}+n_{2}-4tn_{1}, kn_{3}+n_{2}]$, for some $k.$ Denote $i=j+j'.$ 
 By Lemma \ref{independent2}, we have  
$(A(n_{3}-n_{2})\otimes A(j))\xi(f_{1}f_{2}f_{3}f_{4}f_{5})\subseteq E$ where $E=
\sum_{i:i, i+\alpha _{1}+\alpha _{2}\in [kn_{3}+n_{2}-4tn_{1}, (k+1)n_{3}-n_{2}-1],0\leq k'}(A(k'\cdot n_{3}-n_{2})\otimes A(i))\xi(f_{1}f_{2})(A\otimes A).$

 Let $q=\alpha _{1}+\alpha _{2}-5tn_{1},$ then ${\bar Z}_{5tn_{1}}(K\xi(f_{1}f_{2}))\subseteq A(n_{2})\otimes A(q).$ By Lemma \ref{1}, applied for $n=5tn_{1}$ and for the linear space $K\xi(f_{1}f_{2})$, we get  
 \[E\subseteq \sum_{i:i,i+q\in [kn_{3}+n_{2}+1, (k+1)n_{3}-n_{2}-1],0\leq k'}(A(k'\cdot n_{3})\otimes A(i))){\bar Z}_{5tn_{1}}(K\xi(f_{1}f_{2}))(A\otimes A).\] Observe now 
that  $\xi(f_{1}f_{2})\in D$,  so ${\bar Z}_{5tn_{1}}(K\xi(f_{1}f_{2}))\subseteq F'(f, n_{1}, n_{2}, n_{3})$ hence  
$E\subseteq Q(F'(f,n_{1},n_{2},n_{3}),n_{2},n_{3})\subseteq Q$,

{\bf Case 3.}  $j+j'\in [kn_{3}-n_{2}, kn_{3}-n_{1}-1]$, for some $k$. We have the following two subcases.

{\bf Case 3a.} $j+j'\in [kn_{3}-n_{2}+4tn_{1}, kn_{3}-n_{1}-1]$, for some $k$. Denote $i=j+j'.$ 
 Then    
$[i-\alpha _{3}, i]$ is disjoint with any interval $[k'n_{2}-n_{1}, k'n_{2}+n_{1}]$, because $i\leq kn_{3}-n_{1}-1,$ and $i-\alpha _{3}\geq kn_{3}-n_{2}+4tn_{1}-\alpha _{3}>kn_{3}-n_{2}+n_{1}+1.$  
Hence,  $(A(n_{3}-n_{2})\otimes A(j))\xi(f_{1}f_{2}f_{3}f_{4}f_{5})\subseteq \sum _{k}\sum_{i: i,i-\alpha _{3}\in [kn_{3}-n_{2}+n_{1}+1, kn_{3}-n_{1}-1 ] }(A(n_{3})\otimes A(i-\alpha _{3}))\xi(f_{3})(A\otimes A)\subseteq Q(\xi(T(f,n_{1})),n_{1},n_{2})\subseteq Q.$

 {\bf Case 3b.} $j+j'\in [kn_{3}-n_{2}, kn_{3}-n_{2}+4tn_{1}]$, for some $k$. Denote $i=j+j'-\alpha _{3}-\alpha _{4}.$ 
 Then, $i, i+\alpha _{3}+\alpha _{4}\in [kn_{3}-n_{2}-2tn_{2}-2t, kn_{3}-n_{2}+4tn_{1}],$ and so $i, i+\alpha _{3}+\alpha _{4}\in [(k-1)n_{3}+n_{2}+1, kn_{3}-n_{2}+4tn_{1}].$    
 By Lemma \ref{independent2}, applied for the linear space $K\xi(f_{3}f_{4})$, $p=5tn_{1}$  we have 
$(A(n_{3}-n_{2})\otimes A(j))\xi(f_{1}f_{2}f_{3}f_{4}f_{5})\subseteq P_{1}$ where 
 $P_{1}=\sum_{i:i,i+\alpha_{3}+\alpha _{4}\in [ (k-1)n_{3}+n_{2}+1, kn_{3}-n_{2}+4tn_{1}]}(A(n_{3})\otimes A(i))\xi(f_{3}f_{4})(A\otimes A).$
 Denote  $q=\alpha _{3}+\alpha _{4}-5tn_{1},$ then $ Z_{5tn_{1}}(K\xi(f_{3}f_{4}))\subseteq A(n_{2})\otimes A(q).$ 
It follows that $P_{1}\subseteq P_{2}$ where  
\[P_{2}=\sum_{i:i,i+q\in [ (k-1)n_{3}+n_{2}+1, ,kn_{3}-n_{2}-1]}(A(n_{3})\otimes A(i))Z_{5tn_{1}}(Kf_{1}f_{2})(A\otimes A).\]
  Notice that $\xi(f_{3}f_{4})\in D$, and so $ Z_{5tn_{1}}(K\xi(f_{3}f_{4}))\subseteq F'(f, n_{1}, n_{2},n_{3}),$ hence 
$P_{2}\subseteq Q(F(f,n_{1}, n_{2}), n_{2}, n_{3})\subseteq Q$.
\end{proof}

\section{Appendix: Proof of Theorem \ref{S}. Constructing space $S$}\label{appendix4}

 In this section we will show that there is a nil ideal $N$ in $A\otimes A$ such that $N\subseteq S$ for some `good' linear space $S\subseteq A\otimes A$. Space $S$ will be constructed in Corollary $3.5$  in this section. 

 We will start with a supporting lemma which can be read independently of the rest of the paper.  

\begin{lemma}\label{independent2}
 Let $n_{2}, i, i', \alpha , \beta $  be natural numbers such that $4p<n_{2}$, $-p\leq i,i'\leq p$, $0\leq \alpha , \beta \leq p.$ Let $G\subseteq A(n_{2}-\alpha )\otimes A(n_{2}-\beta )$ be a linear space. Then there is set $Z_{i,i' }(G)\subseteq A(n_{2})\otimes A(n_{2})$ with 
 $\dim _{K}Z_{i,i' }(G)<2^{6p+2}\dim_{K} G$ and such that for all $k,k'\geq p$,
 \[(A(k+i)\otimes A(k'+i')) G (A(3p)\otimes A(3p)) \subseteq  (A(k)\otimes A(k'))Z_{i,i' }(G) (A\otimes A).\]
\end{lemma}
\begin{proof}  We have four cases:
{\bf Case $1.$} The case when $i,i'>0$. 
 Denote $E=(A(i)\otimes A(i'))G(A(\alpha )\otimes A(\beta ))$. Then $\dim _{K}E\leq 2^{4p} \dim _{K}G$.
 Note that $E\subseteq A(n_{2}+i)\otimes A(n_{2}+i')$. Similarly, as in Lemma \ref{independent} there is set  $T(E)\subseteq A(n_{2})\otimes A(n_{2})$ such that  $E\subseteq T(E)(A(i)\otimes A(i'))$ and $\dim_{K}(T'(E))\leq 2^{2p+2}\dim _{K}E\leq 2^{6p+2} \dim _{K}G$. 
 We will set then $Z_{i,i'}=T(E)$. Observe that by construction, for every $j, j'\geq 0$,  we have 
 $(A(k+i)\otimes A(k'+i')) G (A(3p)\otimes A(3p)) \subseteq  
(A(k)\otimes A(k'))E(A(3p-\alpha -i)\otimes A(3p-\beta -i'))\subseteq 
(A(k)\otimes A(k'))Z_{i,i'}(A\otimes A)$, as required.

{\bf Case $2.$} The case when $i,i'\leq 0$. 
 Denote $E=\sum _{\alpha ,\beta }G(A(-i+\alpha )\otimes A(-i'+\beta )$. Then $\dim _{K}E\leq 2^{4p} \dim _{K}G$.
 Note that $E\subseteq A(a_{2}-i)\otimes A(a_{2}-i')$ (where $i,i'\leq 0$). 
 By an analogon of Lemma \ref{independent} 
 there is set $T'(E)\subseteq A(a_{2})\otimes A(a_{2})$ such that  $E\subseteq (A(-i)\otimes A(-i'))T'(E)$ and $\dim_{K}(T'(E))\leq 2^{2p+2}\dim _{K}E\leq 2^{6p+2} \dim _{K}G$. 
 We will set then $Z_{i,i'}=T'(E)$. Observe that by construction, for every $j, j'\geq 0$,  
\[A(k+i)\otimes A(k'+i')) G (A(3p)\otimes A(3p)) \subseteq  (A(k)\otimes A(k'))Z_{i,i' }(A\otimes A) ,\] as required.

{\bf Case $3.$} The case when $i>0$, $i'\leq 0$. 
 Denote $E=(A(i)\otimes K)G(A(\alpha )\otimes A(\beta -i')$. Then $\dim _{K}E\leq 2^{4p} \dim _{K}G$.
 Note that $E\subseteq A(a_{2}+i)\otimes A(a_{2}-i')$ (where $i'\leq 0$).
 By a similar argument as in  Lemma \ref{independent} there is set  $T''(E)\subseteq A(a_{2})\otimes A(a_{2})$ such that  $E\subseteq (K\otimes A(-i'))T''(E)(A(i)\otimes K)$ and $\dim_{K}(T''(E))\leq 2^{2p+2}\dim _{K}E\leq 2^{6p+2} \dim _{K}G$. 
 We will set then $Z_{i,i'}=T''(E)$. Observe that by construction, for every $j, j'\geq 0$,  
\[(A(k+i)\otimes A(k'+i')) G (A(3p)\otimes A(3p)) \subseteq  (A(k)\otimes A(k'))Z_{i,i'}(A\otimes A) ,\] as required.

{\bf Case $4.$} The case when $i\leq 0$, $i'>0$. This case is done similarly as case $3$.
\end{proof}

We are now ready to prove the following theorem.
\begin{theorem}\label{important2} Let  $ f \in \sum_{0<i,j\leq t}A(i)\otimes A(j)$ for some $t  \geq 1$. Let  $n_{1}, n_{2},n_{3}, m$ be natural numbers such that $n_{1}$ divides $n_{2}$, $n_{2}$ 
divides $n_{3}$, $10tn_{1}<n_{2},$ $10tn_{2}+2n_{3}<n_{3}$, $10tn_{3}+2n_{3}<m$, $10t<n_{1}$. Let  $T(f,n_{1})$, $E'(f,n_{1}, n_{2})$,  $F(f,n_{1}, n_{2}, n_{3})$, $F'(f,n_{1}, n_{2}, n_{3})$ be as in Lemmas \ref{final}, \ref{smaller} and \ref{finalop}.
 Then there exists a homogeneous  linear space  
$G(f,n_{1}, n_{2})\subseteq A(n_{2})\otimes A(n_{2})$  with $\dim _{K} G(f, n_{1}, n_{2})<n_{2}^{3} 2^{500tn_{1}}$ 
and such that, for every integer $m >20n_{3}$,  the ideal generated by $f^{m}$ and $\xi(f^{m})$ in $A\otimes A$  is contained in the right ideal of $A\otimes A$,  $M+Q_{1}+Q_{2}+Q_{3}$ where  $Q_{1}=Q(T(f,n_{1})+\xi(T(f,n_{1})),n_{1}, n_{2})$, $Q_{2}=Q(F(f,n_{1}, n_{2}, n_{3}),n_{2}, n_{3})+Q'(F'(f,n_{1}, n_{2}, n_{3}), n_{2}, n_{3})$,  $Q_{3}=Q'(T'(f,n_{1})+\xi(T'(f,n_{1})), n_{1}, n_{2})$ and where  $M=\sum_{i=1, 2, \ldots ;j=1,2,\ldots }((A(i\cdot n_{3})+A(i\cdot n_{3}-n_{2}))\otimes (A(j\cdot n_{3})+A(j\cdot n_{3}-n_{2})))G(f, n_{1},n_{2})(A\otimes A).$
\end{theorem}
\begin{proof} 
Let $I'$ be the ideal generated in $A\otimes A$ by $\xi(f^{m-2n_{3}})$ and $f^{m-2n_{3}}$, and $I$ be the ideal generated by 
$\xi(f^{m})$ and $f^{m}.$ Then $I\subseteq I'(A(n_{3})\otimes A(n_{3})),$ since $\xi(f^{m})=\xi (f^{n_{3}}\cdot \xi(f^{m-2n_{3}})\cdot f^{n_{3}}).$   
 So we only need to show that $I'(A(n_{3})\otimes A(n_{3}))\subseteq M+Q_{1}+Q_{2}+Q_{3}$.

 By Lemmas   \ref{final} and \ref{finalop}, the two sided ideal generated by $f^{m-2n_{3}}$ and $\xi(f^{m-2n_{3}})$ belongs to the right ideal of $A\otimes A$ generated by elements from $Q_{1}$, $Q_{2}$ and from $P+P'$ where 
 $P=   \sum_{k=1}^{\infty }\sum_{j'\in U(n_{1}, n_{3})}(A(k\cdot n_{3})\otimes A(j'))(E'(f,n_{1}, n_{2}),$ and from 
 $P'=\sum_{ \alpha=0}^{ \infty }\sum_{k=1}^{\infty }
\sum_{j':j'+\alpha \in U(n_{1}, n_{3})}(A(k\cdot n_{3})\otimes A(j))\xi(E'(f,n_{1}, n_{2},\alpha )).$ 
 So we only need to show that $(P+P')(A(n_{3})\otimes A(n_{3}))\subseteq M+Q_{1}+Q_{2}+Q_{3}$.

 Let $G_{1}(f,n_{2}-n_{1})$, $G_{2}(f,n_{2}-n_{1})$ be as in Theorem \ref{G} applied for $n=n_{2}-n_{1}$, then we can write 
 $G_{1}(f,n_{2}-n_{1})= \sum_{\gamma \geq {n_{2}-n_{1}\over 2t}}^{n_{2}-n_{1}} G_{\gamma }$ 
 for some  $G_{i}\subseteq A(n_{2}-n_{1})\otimes  A(i).$ Let $c$ be the smallest number such that  \[c\geq  {n_{2}-n_{1}\over 2t}.\]

 Write now $G_{1}(f,n_{2}-n_{1})=G+G'$ where $G=\sum_{i=c}^{ n_{2}-10tn_{1}}G_{i}$, $G'=
\sum_{i= n_{2}-10tn_{1}+1}^{n_{2}-n_{1}}G_{i}.$ Similarly, write 
  $G_{2}(f,n_{2}-n_{1})={\bar G}+{\bar G}'$ where ${\bar G}=\sum_{i=c}^{ n_{2}-10tn_{1}}{\bar G}_{i}$, ${\bar G}'=
\sum_{i= n_{2}-10t_{1}+1}^{n_{2}-n_{1}}{\bar G}_{i},$ for suitable  ${\bar G}_{i}\subseteq A(i)\otimes A(n_{2}-n_{1})$.

Define now  $F_{1}(f,n_{1}, n_{2})=\sum_{j= n_{2}-10tn_{1}}^{n_{2}-n_{1}}(D_{j}+{\bar D}_{j})$ where \[D_{j}= 
\sum _{-p\leq i,i'\leq p}Z_{i,i' }(G_{j}), {\bar D}_{j}= 
\sum _{-p\leq i,i'\leq p}Z_{i,i' }({\bar G}_{j}),\]  
 where $Z_{i,i' }(G_{j})$ is as in Lemma \ref{independent2} applied for  $p =12t n_{1}$, for  the same $n_{2}$, and  for $G=G_{j}$ (and
 $Z_{i,i' }({\bar G}_{j})$ for $G={\bar G}_{j}$). 
By Lemma \ref{independent2}, we have $F_{1}(f,n_{1}, n_{2})\subseteq A(n_{2})\otimes A(n_{2})$ and $\dim _{K}F_{1}(f,n_{1}, n_{2})<n_{2}^{3} 2^{500tn_{1}-1}$.

Denote $L_{j'}=\sum_{j\in U(n_{1}, n_{3})}(A(j'n_{3})\otimes A(j))(G+{\bar G})(A(n_{3})\otimes A(n_{3})).$ Denote $p=12tn_{1}$. 
 Observe now that if $j\in U(n_{1}, n_{3})$ then there exists $k', i'$ such that  $-p\leq i'\leq p$, $k'$ is divisible by $n_{3}$ and $j=k'+i'$. 
Apply  Lemma \ref{independent2} for this $k',i'$ and $p$, and for  for $k=j'n_{3}$, $i=0$.  
 We then get for every $j',$ $L_{j'}\subseteq 
(A(k)\otimes A(k'))G_{1}(f, n_{1}, n_{2})(A\otimes A)\subseteq 
M,$ provided that $F_{1}(f,n_{1},n_{2})\subseteq G(f,n_{1}, n_{2}).$

Similarly, define 
$F_{2}(f,n_{1}, n_{2})
=\sum_{j= n_{2}-(t+1)n_{1}}^{n_{2}-n_{1}}(D'_{j}+{\bar D}'_{j})$ where \[D'_{j}= 
\sum _{-p\leq i,i'\leq p}Z_{i,i'}(\xi(G_{j})), {\bar D}'_{j}= 
\sum _{-p\leq i,i'\leq p}Z_{i,i' }(\xi({\bar G}_{j})),\]  
 where $Z_{i,i'}(\xi(G_{j}))$ is as in Lemma \ref{independent2} applied for $p =12t n_{1}$, for  the same $n_{2}, n_{3}$, and  for $G=\xi (G_{j})$ (and $Z_{i,i' }(\xi({\bar G}_{j}))$ for $G=\xi({\bar G}_{j})$).
 Similarly as before, $F_{2}(f,n_{1},n_{2})\subseteq A(n_{2})\otimes A(n_{2})$ and  
  $\dim _{K}F_{2}(f,n_{1}, n_{2})\leq n_{2}^{3} 2^{500tn_{1}-1}.$\\
For $g\in \xi(G)+\xi({\bar G})$ denote ${\bar L}_{j',g}=\sum_{j:j+r(g)\in U(n_{1}, n_{3})}(A(j'n_{3})\otimes A(j))g(A(n_{3})\otimes A(n_{3})).$ Denote $p=12tn_{1}$. Observe now that since $j+r(g)\in U(n_{1}, n_{2})$ and  $g\in \sum_{l,l'=n_{2}-10tn_{1}}^{n_{2}-n_{1}}A(l)\otimes A(l')$ then there exists 
 $k'$ such that $k'+n_{2}$ is divisible by $n_{3}$, and $-p\leq i'\leq p,$ such that $j=k'+i'.$
Apply  Lemma \ref{independent2} for this $k',i'$ and $p$, and for  for $k=j'n_{3}$, $i=0$.  
 We then get for every $j',g'$, $L_{j',g}\subseteq  
(A(k)\otimes A(k'))F_{2}(f, n_{1}, n_{2})(A\otimes A)\subseteq M,$ 
 provided that $F_{2}(f,n_{1},n_{2})\subseteq G(f,n_{1}, n_{2})$.
 Let now $h$ be a homogeneous element such that $h\in g'(A(l)\otimes A(l'))$ for some $g'\in G+\bar G$ and   for some $l,l'$, and denote $g=\xi (g')$.  Observe that then  
 $\xi (h)\in (K\otimes A(l')g(A(l)\otimes K)$. Therefore,   
$ \sum_{j:j+r(h)\in U(n_{1}, n_{3})}(A(j'n_{3})\otimes A(j))\xi(h)(A(n_{3})\otimes A(n_{3}))\subseteq 
 \sum_{j:j+l'+r(g)\in U(n_{1}, n_{3})}(A(j'n_{3})\otimes A(j+l'))\xi(g)(A(n_{3})\otimes A(n_{3}))(A\otimes A)
\subseteq M,$ for some $k''$ divisible by $a_{n_{3}}$, 
 provided that $F_{2}(f,n_{1},n_{2})\subseteq G(f,n_{1}, n_{2}),$ as before (with $j+l'$ instead of $j$). 

We now set $G(f,n_{1}, n_{2})=F_{1}(f,n_{1}, n_{2})+F_{2}(f,n_{1}, n_{2}).$

 Recall that $E'(f, n_{1}, n_{2})=S(f,n_{1}, n_{2})+S'(f,n_{1}, n_{2}),$ as in Lemma \ref{E}.

 We will  now  show that $(P+P')(A(n_{3})\otimes A(n_{3})\subseteq  M+Q_{1}+Q_{2}+Q_{3}.$
 by considering all the possible cases :

{\bf Case 1.} To show that $\sum_{j'=1}^{\infty }\sum_{j'\in U(n_{1}, n_{3})}(A(j'\cdot n_{3})\otimes A(j))S(f,n_{1}, n_{2})(A(n_{3})\otimes A(n_{3}))\subseteq M+Q_{1}+Q_{2}.$
 Notice that $S(f, n_{1}, n_{2})=G_{1}(f,n_{2}-n_{1})T(f,n_{1})=GT(f,n_{1})+G'T(f,n_{1}),$ so we consider the following subcases

{\bf Case 1a}    To  prove that for all $j'> 0,$ we have   $\sum_{j\in U(n_{1},n_{3})}(A(j'n_{3}\otimes A(j)) GT(f,n_{1})(A(n_{3})\otimes A(n_{3}))\subseteq M$, this follows from the first part of this proof.
 
{\bf Case 1b}    To  prove that for all $j'> 0,$ we have   $\sum_{j\in U(n_{1},n_{3})}(A(j'n_{3}\otimes A(j)) G'T(f,n_{1})\subseteq Q_{1}$.
By assumption $c\geq  {n_{2}-n_{1}\over 2t}>2n_{1} $. Observe that  since $c\leq  \beta \leq n_{2}-10tn_{1}$, and $j\in U(n_{1}, n_{3})$ then 
 interval $[j+\beta, j+\beta +n_{1} ]$ is disjoint with any interval $[kn_{2}-n_{1}, kn_{2}+n_{1}]$. It follows that 
then  $\sum_{j\in U(n_{1},n_{3})}(A(j'n_{3}\otimes A(j)) G_{\beta}T(f,n_{1})\subseteq  \sum_{j\in U(n_{1},n_{3})}(A(j'n_{3}+n_{2}-n_{1})\otimes A(j+\beta ))T(f,n_{1})\subseteq  Q(T(f,n_{1}),n_{1}, n_{2})\subseteq Q_{1}.$ 

{\bf Case 2} To show that for all $j'>0$, $\sum_{1\leq j', \alpha }\sum_{j+\alpha \in U(n_{1}, n_{3})}(A(j'\cdot n_{3})\otimes A(j))(\xi(S(f,n_{1}, n_{2})\cap (A\otimes A(\alpha ))(A(n_{3})\otimes A(n_{3}))\subseteq M+Q_{1}+Q_{2}.$
 Notice that $\xi(S(f, n_{1}, n_{2}))=\xi(GT(f,n_{1})+G'T(f,n_{1})),$ so we consider the following subcases

{\bf Case 2a} To  show that  $L_{j'}\subseteq M$ for all $j' >0$,  and all homogeneous $h\in \xi(GT(f,n_{1})(A\otimes A))$ where \[L_{j' ,h}= \sum_{j+r(h) \in U(n_{1}, n_{3})}(A(j'\cdot n_{3})\otimes A(j))\xi(h)(A(n_{3})\otimes A(n_{3})).\]
 Observe that  $L_{j',h }\subseteq M$ by the first part of this proof.

{\bf Case 2b}  To  prove that $L_{j',g}\subseteq Q_{1}$ for every every $j'>0,$ and every  homogeneous  $g\in \xi(GT(f,n_{1})(A\otimes A))$ where $L_{j',g}=\sum_{j: j+r(g)\in U(n_{1}, n_{3})}(A(j'n_{3})\otimes A(j))g.$ Recall that $r(g)$ is the right degree of $g,$ so $r(g)=i$ if $g\in A\otimes A(i)$.
  Observe that if  $g=\xi(g_{1}g_{2}g_{3})$ where $g_{1}\in G_{\beta }, g_{2}\in T(f,n_{1})$, $g_{3}\in A\otimes A$ are homogeneous elements, then   $L_{j',g} \subseteq \sum_{ l:l+\beta +r(g_{2})\in U(n_{1}, n_{3})}(A(j'n_{3}+n_{2}-n_{1})\otimes A(l))\xi (g_{2})(A\otimes A),$ (where $l=j+r(g_{3})$).
 
 By assumption $2n_{1}<c\leq  {n_{2}-n_{1}\over 2t}$ and   $c\leq  \beta \leq n_{2}-10tn_{1}.$ Recall also that $r(g_{2})\leq 2tn_{1}+t,$
 by Lemma \ref{introducingT}. Therefore $l+\beta +r(g_{2})\in U(n_{1}, n_{3})$ yields that  the interval $[l,l+r(g_{2})]$ is disjoint with any interval $[kn_{2}-n_{1}, kn_{2}+n_{1}].$  
  It follows that $L_{j',g}\subseteq Q(\xi(T(f,n_{1})),n_{1}, n_{2})\subseteq Q_{1}.$ 

{\bf Case 3}  To show that for all $j'>0,$  $\sum_{j'=0}^{\infty }\sum_{j\in U(n_{1}, n_{3})}(A(j'\cdot n_{3})\otimes A(j))S'(f,n_{1}, n_{2})(A(n_{3})\otimes A(n_{3}))\subseteq M+Q_{1}+Q_{2}+Q_{3}.$

  By Lemma \ref{E},  $S'(f, n_{1}, n_{2})= G_{2}(f,n_{2}-n_{1})S=\bar {G}S+ \bar {G}'S,$ where $S=\bigcap_{i=1}^{2n_{1}}T'(f,i)T(f,n_{1})(A\otimes A),$ so we consider the following subcases 

{\bf Case 3a} To show that for all $j'>0,$  $\sum_{j'=0}^{\infty }\sum_{j\in U(n_{1}, n_{3})}(A(j'\cdot n_{3})\otimes A(j)){\bar G}S(A(n_{3})\otimes A(n_{3}))\subseteq M.$ This follows from the first part of this proof where we showed that 
 $\sum_{j\in U(n_{1}, n_{3})}(A(j'n_{3})\otimes A(j))(G+{\bar G})(A(n_{3})\otimes A(n_{3}))\subseteq M.$

{\bf Case 3b} To show that for all $j'>0,$ $L_{j'}\subseteq Q_{3}$, where  $L_{j'}=\sum_{j\in U(n_{1}, n_{3})}(A(j'\cdot n_{3})\otimes A(j)){\bar G}'S(A(n_{3})\otimes A(n_{3})).$
 Let $k'$ be divisible by $n_{3}$ and such that $-n_{1}\leq k'-j\leq n_{1}$ (because $j\in U(n_{1}, n_{2})$ such $k'$ exists). Notation $k'\in U$ will mean that $k'$ is divisible by $n_{3}.$
 
Denote $\alpha =n_{1}+k'-j,$ then $\alpha +(n_{2}-n_{1}+j)=k'+n_{2},$ and $0\leq \alpha \leq 2n_{1}.$
  We get  
$L_{j'}\subseteq \sum_{\beta =c}^{n_{2}-10tn_{1}} J_{\beta }$ where $J_{\beta }=((A(j'\cdot n_{3})\otimes A(j)){\bar G}_{\beta }T'(f,\alpha  )T'(f,n_{1})(A\otimes A).$  Recall that $T'(f,\alpha )\subseteq \sum_{i=1}^{2t\alpha +t}A(i)\otimes A(n_{1}),$
(by Lemmas \ref{introducingT'}, \ref{smaller})
 hence \[J_{\beta }\subseteq \sum_{i=1}^{2t\alpha +t}\sum_{k'\in U}(A(j'\cdot n_{3}+\beta +i)\otimes A(k'+n_{2}))T'(f, n_{1})(A\otimes A)\] (because 
$k'+n_{2}=j+(n_{2}-n_{1})+\alpha )$. 
 Recall that $2n_{1}<c\leq \beta \leq n_{2}-10tn_{1}, \alpha <2n_{1},$ $T'(f,n_{1})\subseteq \sum_{i'=1}^{2tn_{1}+t}A(n_{1})\otimes A(i').$  Observe that 
  for $i\leq 4tn_{1}+t, i'\leq 2tn_{1}+t$, the interval $[j'\cdot n_{3}+\beta+i, j'\cdot n_{3}+\beta+i+i']$ is disjoint with any interval $[kn_{2}-n_{1}, kn_{2}+n_{1}].$ It follows that $J_{\beta }\subseteq Q'(T'(f, n_{1}), n_{1}, n_{2})\subseteq Q_{3}.$


{\bf Case 4}   To show that for all $j'>0,$  $L_{j'}\subseteq Q_{3}$ for every homogeneous  $g\in \xi(S'(f, n_{1}, n_{2}))$  where $L_{j',g}=\sum_{j: j+r(g)\in U(n_{1}, n_{3})}(A(j'n_{3})\otimes A(j))g(A(n_{3})\otimes A(n_{3})).$
 (recall that $r(g)$ is the right degree of $g$). 

 By Lemma \ref{E},  $S'(f, n_{1}, n_{2})= G_{2}(f,n_{2}-n_{1})S=\bar {G}S+ \bar {G}'S,$ where $S=\bigcap_{i=0}^{2n_{1}}T'(f,i)T(f,n_{1})(A\otimes A).$ So we consider the following subcases 

{\bf Case 4a} To show that for all $j'>0,$ and any $h\in {\bar G}S$ we have  
 $L_{j',h}\subseteq M$ where $L_{j',h}=\sum_{j:j+r(h)\in U(n_{1}, n_{3})}(A(j'\cdot n_{3})\otimes A(j))\xi(h)(A(n_{3})\otimes A(n_{3}))\subseteq M.$ This follows from the first part of this proof.

{\bf Case 4b} To  prove that $L_{j',g}\subseteq Q_{3}$ for homogeneous  $g\in {\bar G}'S$ and  every $j'>0,$   where $L_{j',g}=\sum_{j: j+r(g)\in U(n_{1}, n_{3})}(A(j'n_{3})\otimes A(j))\xi(g).$
Since $j+r(g)\in U(n_{1}, n_{3})$, then  there is $k'$ is divisible by $n_{3}$ and such that $-n_{1}\leq j+r(g)-k'\leq n_{1}.$ Denote $\alpha =n_{1}+ j+r(g)-k',$ then  
$0\leq \alpha \leq 2n_{1}$.

  Since $g\in {\bar G}'S$ then $g\in {\bar G}_{\beta }T'(f,\alpha )T(f,n_{1})(A\otimes A),$ for some $c\leq \beta \leq n_{2}-10tn_{1}.$
 Therefore   $g=g_{1}g_{2}g_{3}g_{4}$ where $g_{1}\in {\bar G}_{\beta },$ $g_{2}\in T'(f, \alpha ),$ $g_{3}\in T'(f, n_{1}),$ $g_{4}\in  A\otimes A$ are homogeneous elements.
 It follows that 
 $L_{j', g}\subseteq \sum_{j: j+(r(g_{1})+r(g_{2})+r(g_{3})+r(g_{4})) \in U(n_{1}, n_{3})}((A(j'\cdot n_{3}+\beta +l(g_{2}))\otimes A(j+r(g_{4}))\xi(g_{3})(A\otimes A).$  Notice that $r(g_{1})+r(g_{2})=n_{2}-n_{1}+\alpha .$

 Recall that $g_{2}\in T'(f, \alpha )\subseteq \sum_{i=1}^{2t\alpha +t }A(i)\otimes A(\alpha ).$  
 Consequently,  $L_{j', g}\subseteq \sum_{i=1}^{2t\alpha +t}\sum_{j: j+n_{2}-n_{1} +\alpha +r(g_{3})+r(g_{4})) \in U(n_{1}, n_{3})}((A(j'\cdot n_{3}+\beta +i)\otimes A(j+r(g_{4}))\xi(g_{3})(A\otimes A).$ 
 Observe that $j+r(g_{3})+r(g_{4})=j+r(g)-(r(g_{2})+r(g_{1}))=j+r(g)-\alpha -n_{2}+n_{1}=k'-n_{2}.$ Denote $l=j+r(g_{4})$ and let $k'\in U$ mean that $k'$ is divisible by $n_{3}$. With this notation we have  
$L_{j',g}\subseteq \sum_{i=1}^{2t\alpha +t}\sum_{k',l:k'\in U,l+r(g_{3})=k'-n_{2}}((A(j'\cdot n_{3}+\beta +i)\otimes A(l))\xi(g_{3})(A\otimes A).$

 Recall that $2n_{1}<c<\beta \leq n_{2}-10tn_{1}$ and $\alpha <2n_{1},$ $g_{3}\subseteq \sum_{i'=1}^{2tn_{1}+t} A(i')\otimes A(n_{1}),$ $\alpha \leq 2n_{1}.$ Hence for $i\leq 2t\alpha +t, i'\leq 2tn_{1}+2t$ the interval $[\beta+i, \beta +i+i']$ is disjoint with any interval $[kn_{2}-n_{1}, kn_{2}+n_{1}].$
   Recall that  and that $n_{3}$ is divisible by $n_{2}, $ so  interval $[j'n_{3}+\beta+i, j'n_{3} +\beta +i+i']$ is disjoint with any interval $[kn_{2}-n_{1}, kn_{2}+n_{1}].$

Therefore, $L_{j',g'}\subseteq Q'(\xi(T'(f,n_{1})),n_{1}, n_{2})=Q_{3}$, as required.
 \end{proof} 

 We are now ready to prove Theorem \ref{S}.

{\bf Proof of Theorem \ref{S}.} If $K$ is a countable field then we can ennumerate elements of $A\otimes A$, as $f_{1}, f_{2}, $
 in a such way that: all elements of $A\otimes A$ are listed; some $f_{i}$ can be zero; $f_{i}\subseteq \sum_{i,j\leq i-1}A(i)\otimes A(j)$ for all $i\geq 2$, and moreover $f_{1}=f_{2}=0$. Let $N$ be the ideal generated by $f_{i}^{20ia_{i+1}}$ and $\xi (f_{i}^{20ia_{i+1}})$ for all $i\geq 2.$  By Theorem \ref{important2}, applied for $f=f_{n}$, $n_{1}=a_{n-1}$, $n_{2}=a_{n}$, $n_{3}=a_{n+1}$, $t=n-1$,  the two sided ideal generated by  $f_{n}^{100a_{n+2}}$ and 
 $\xi(f_{n}^{100a_{n+2}})$ is contained in $M_{n+1}$+ $Q_{n}+Q_{n-1}+Q'_{n-1}$, where linear spaces $F_{n}, F'_{n}$ and $G_{n}$ are defined as follows.
  With notation as in  Theorem \ref{important2}, for each $n\geq 1$ denote 
\[F_{n}=F(f_{n}, a_{n-1},a_{n}, a_{n+1})+F'(f_{n}, a_{n-1}, a_{n}, a_{n+1})+T(f_{n+1}, a_{n})+\xi(T(f_{n+1}, a_{n})),\]
\[ F'_{n}=T'(f_{n+1}, a_{n})+\xi(T'(f_{n+1}, a_{n})), G_{n}=G(f_{n}, a_{n-1}, a_{n}).\] We will show that   
  $F_{n}$, $F'_{n},$  $G_{n}$ satisfy assumptions of our theorem.  By Theorem \ref{important2}, 
$G_{n}\subseteq A(a_{n})\otimes A(a_{n})$ and 
$\dim _{K}G_{n}< a_{n}^{3} 2^{500na_{n-1}}<a_{n+2}^{100n}2^{500na_{n-1}}.$
  Observe that by Lemmas \ref{final} and  \ref{finalop}, we have 
$\dim _{K}F(f_{n}, a_{n-1},a_{n}, a_{n+1}),$ $ \dim _{K} F'(f_{n}, a_{n-1}, a_{n}, a_{n+1})\leq 2^{40n\cdot a_{n-1}}a_{n}$ and 
 By Lemmas \ref{introducingT}, \ref{introducingT'} and \ref{smaller}, we have 
 $\dim _{K} T(f_{n+1}, a_{n}),$ $\dim _{K} T'(f_{n+1}, a_{n})\leq  max( 2^{4n+3}a_{n}^{3}n^{3}, 2^{2n^{2}+1})$ 
 where $max (a,b)$ denotes the maximum of $a$ and $b$.
 Therefore $\dim _{K} F_{n}, \dim _{K} F'_{n}\leq a_{n+2}^{100n}2^{500n\cdot a_{n-1}}.$
 By Lemmas \ref{final} and \ref{finalop}, and since $  a_{n}>100na_{n-1}$ for every $i$,   
 we have $F(f_{n}, a_{n-1},a_{n}, a_{n+1}), F'(f_{n}, a_{n-1}, a_{n}, a_{n+1})\subseteq A(a_{n})\otimes  \sum_{j>{a_{n}\over 10n}}^{10na_{n}}A(j).$ By Lemmas \ref{introducingT}, \ref{introducingT'} and \ref{smaller}, we have 
$T(f_{n+1}, a_{n})\subseteq A(a_{n})\otimes \sum_{i\geq {a_{n}-n\over 2n}}^{2na_{n}+n}A(i)$ 
  and $T'(f_{n+1}, a_{n}) \subseteq \sum_{i\geq {a_{n}-n\over 2n}}^{2na_{n}+n}A(i)\otimes A(a_{n}).$  It follows that   $F_{n}\subseteq A(a_{n})\otimes  \sum_{j>{a_{n}\over 10n}}^{10na_{n}}A(j)$ and $F'_{n}\subseteq \sum_{j>{a_{n}\over 10n}}^{10na_{n}}A(j)\otimes A(n),$ since ${n<{a_{n}\over 2}}.$ 
Notice that we can put $G_{1}=0,$ $G_{2}=0,$ $F_{1}=0$ because $f_{1}, f_{2}=0.$

 Assume now that $K$ is an uncountable field, as $A$ has two generators , it follows that set $M\otimes M$ is countable, where $M\subseteq A$ is a set of monomials in $A$.  
 Therefore, the set of finite subsets of $M\otimes M$ is also countable. If $H$ is a subset of $M\otimes M$ by $KH$ we will mean the set of all linear combinations of elements from $H$.

 Let $H_{2}, H_{3}, \ldots $ be finite subsets of $M\otimes M$ such that for each $i\geq 2$, the cardinality of $H_{i}$ is less than $i$ and  
 $H_{i}\subseteq \sum_{i,j\leq i-1}A(i)\otimes A(j)$ and for all $i\geq 2$. Moreover some sets $H_{i}$ may be zero.
 Observe that we can list all finite subsets of $M\otimes M$ as sets $H_{2}, H_{3}, \ldots ,$ because we allow some sets $H_{i}$ to be zero. Moreover, we can assume that $H_{1}=H_{2}=H_{3}=\emptyset.$

For a given set $H$ and a natural number $n$ let $P_{n}(H) =\{f^{n}: f\in H\}.$ Then $K\cdot P_{n}(H)$ is the linear space spanned by elements from $P_{n}(H).$

 We now claim that for every $n\geq 2$ there is a set ${\bar H}_{n}\subseteq KH_{n}$ and such that the cardinality of ${\bar H}_{n}$ is less than  $(20na_{n+1})^{n}$ and moreover $K\cdot P_{20na_{n+1}}({\bar H}_{n})=K\cdot P_{20na_{n+1}}(KH_{n})$. To see that  let $f\in K\cdot H_{n}$, then $f=\sum_{i=1}^{n}\alpha _{i}m_{i}$ where $\alpha _{i}\in K$, $m_{i}\in H$. Let $w(j_{1}, j_{2}, \ldots ,j_{n})$ be the sum of all products of length $20na_{n+1}$ of elements from the set $H$ which are products of exactly $j_{i}$ elements $m_{i}$ for every $i\leq n$. Then $f^{20na_{n+1}}$ is a linear combination of elements $w(j_{1}, j_{2}, \ldots ,j_{n})$ for various $j_{1}, j_{2}, \ldots , j_{n},$ with $j_{1}+\ldots +j_{n}=20na_{n+1}.$ Because there are less than $(20na_{n+1})^{n}$ such choices of elements $i_{1}, i_{2}, \ldots ,i_{n}$, it follows that the dimension of the linear space $P_{20na_{n+1}}(KH_{n})$ is less than $(20na_{n+1})^{n}$. Therefore, there is a subset ${\bar H}_{n}\subseteq KH_{n}$ with cardinality less than  $(20na_{n+1})^{n}$ and such that $K\cdot P_{20na_{n+1}}({\bar H}_{n})=
K\cdot P_{20na_{n+1}}(KH_{n}),$ as required. Notice than then also $K\cdot \xi (P_{20na_{n+1}}({\bar H}_{n}))=
K\cdot \xi ( P_{20na_{n+1}}(K H_{n}),$

  We will now proceed in a similar way as in the case of countable fields at the beginning of this proof. We can 
 assume that  $H_{1}=H_{2}=H_{3}=\emptyset $ and put $F_{1}=0$, $G_{1}=0$, $G_{2}=0$. 
 Observe that  by Theorem \ref{important2} applied for  $f\in {\bar H}_{n}$, $n_{1}=a_{n-1}$, $n_{2}=a_{n}$, $n_{3}=a_{n+1}$, $t=n-1$,  the two sided ideal generated by  elements from  the set $\{f^{100a_{n+2}}:f\in {\bar H}_{n}\}$ and 
  from the set $\{\xi(f)^{100a_{n+2}}: f\in {\bar H}_{n}\}$ is contained in $M_{n+1}$+ $Q_{n}+Q_{n-1}+Q'_{n-1}$, where linear spaces  $F_{n}, F'_{n}$ and $G_{n}$ are defined as follows.
   $F_{1}=0$, $G_{1}=0$, $G_{2}=0$, and $F_{n}=\{F(f, a_{n-1},a_{n}, a_{n+1})+F'(f, a_{n-1}, a_{n}, a_{n+1})+T(g, a_{n})+\xi(T(g, a_{n})):f\in {\bar H}_{n}, g\in {\bar H}_{n+1}\}.$ Moreover, 
\[ F'_{n}=\{T'(g,a_{n})+\xi(T'(g, a_{n}): g\in {\bar H}_{n+1},\}, \]
\[G_{n}=\{G(f, a_{n-1}, a_{n}): f\in {\bar H}_{n}\}.\]
 We will show that   
  $F_{n}$, $F'_{n},$  $G_{n}$ satisfy the assumptions of our theorem.  By Theorem \ref{important2},  
$G_{n}\subseteq A(a_{n})\otimes A(a_{n})$ and for every $g\in {\bar H}_{n+1},$ 
$\dim _{K} G(g, a_{n-1}, a_{n})<   a_{n}^{3} 2^{500na_{n-1}}.$
 Hence, $\dim _{K} G_{n}<  (a_{n}^{3} 2^{500na_{n-1}})\cdot (20(n+1)a_{n+2}^{n+1})<a_{n+2}^{100n}2^{500na_{n-1}}$ 
 (because  $\dim _{K}{ \bar H}_{n+1}\leq 20(n+1)a_{n+2}^{n+1}$).
  By Lemmas \ref{final}, \ref{smaller} and \ref{finalop}, we have for each $f\in {\bar H}_{n}$, 
$\dim _{K}F(f, a_{n-1},a_{n}, a_{n+1}), \dim _{K} F'(f, a_{n-1}, a_{n}, a_{n+1})\leq 2^{40n\cdot a_{n-1}}a_{n}$ and 
 By Lemmas \ref{introducingT}, \ref{introducingT'} and \ref{smaller}, we have for every $g\in {\bar H}_{n+1},$ 
 $\dim _{K} T(g, a_{n}),$ $\dim _{K} T'(g, a_{n})\leq  max( 2^{4n+3}a_{n}^{3}n^{3}, 2^{2n^{2}+1}).$ Recall that  
$\dim _{K} {\bar H}_{n},$ $\dim _{K}{ \bar H}_{n+1}\leq 20(n+1)a_{n+2}^{n+1}$,  it follows that  
 $\dim _{K} F_{n}, \dim _{K} F'_{n}\leq a_{n+2}^{100n}2^{500n\cdot a_{n-1}}.$
 Similarly as in the countable field case,   $F_{n}\subseteq A(a_{n})\otimes  \sum_{j>{a_{n}\over 10n}}^{10na_{n}}A(j)$ and $F'_{n}\subseteq \sum_{j>{a_{n}\over 10n}}^{10na_{n}}A(j)\otimes A(n).$
 Therefore, sets $F_{n}, F'_{n}, G_{n}$ satisfy thesis of our theorem.\\

Let $N$ be the ideal generated in $A\otimes A$ by sets $P_{20na_{n+1}}({\bar H}_{n})$, for $n=2, 3, $. 
  By the above reasoning, $N\subseteq \sum_{i=1}^{\infty }M_{i}+Q_{i}.$
Moreover, if $f\in A\otimes A$ then $A\in H_{i}$ for some $i$. Then $f^{20ia_{i+1}}\in  P_{20ia_{i+1}}(KH_{i})\subseteq K\cdot P_{20ia_{i+1}}({\bar H}_{i})\subseteq N.$ \\
 Similarly, $\xi (f^{20ia _{i+1}})\in \xi ( P_{20ia_{i+1}}(KH_{i}))\subseteq K\cdot \xi( P_{20ia_{i+1}}({\bar H}_{i}))\subseteq N.$

$ $\\
{\bf Acknowledgements} This research was supported by ERC Advanced grant Coimbra 320974. The author would like to thank  Adam Chapman, Natalia Iyudu, Jerzy Matczuk, Uzi Vishne and Micha{\l} Ziembowski   for their comments on the paper's introduction. The author is also grateful to Pace Nielsen, as  a  simple exposition from his paper \cite{n} was used by her to simplify the
proof of Theorem \ref{S}.  The author is especially grateful to  Andre  Leroy and the organisers of  ``Nord Pas de Calais/Belgium
congress of Mathematics'' for their hospitality during her stay in Valenciennes, Mons and Lens, in October and  November 2013 where the first  part of this paper was prepared.

\end{document}